\documentclass[a4paper, 10pt]{amsart}

\usepackage[square, numbers, comma]{natbib} 

\usepackage[usenames,dvipsnames,svgnames,table]{xcolor}
\usepackage{amsmath,amsthm,amssymb,amssymb,esint,verbatim,tabularx,graphicx}
\usepackage{dsfont}
\usepackage{bbm}
\usepackage{fancyhdr}
\usepackage{enumerate}
\usepackage{hyperref}
\usepackage{cleveref}
\usepackage{amsmath}
\usepackage{fancyhdr}
\usepackage{epic}
\usepackage{pgf,tikz}
\usetikzlibrary{arrows}

\usepackage[utf8]{inputenc}
\usepackage{color}
\usepackage{hyperref}
\usepackage{verbatim}
\usepackage{pdfpages}
\usepackage{empheq}
\usepackage{microtype}
\usepackage[toc,page]{appendix}

\usepackage{mathtools}

\usepackage[inner=2.2cm,outer=2.2cm,bottom=2.6cm,top=2.6cm]{geometry}

\newtheorem{theorem}{Theorem}[section]

\newtheorem{lemma}[theorem]{Lemma}
\newtheorem{proposition}[theorem]{Proposition}
\newtheorem{corollary}[theorem]{Corollary}
\newtheorem{definition}{Definition}[section]

\theoremstyle{remark}
\newtheorem{remark}[theorem]{Remark}
\theoremstyle{definition}

\numberwithin{equation}{section}

\newcommand{\R}{\ensuremath{\mathbb{R}}}
\newcommand{\N}{\ensuremath{\mathbb{N}}}

\newcommand{\veps}{\varepsilon}

\newcommand{\plap}{\ensuremath{\Delta_p}}
\newcommand{\nplap}{\ensuremath{\Delta_p}^{\textup{N}}}

\newcommand{\dd}{\,\mathrm{d}}

\begin{document}

\title[Asymptotic mean value properties for non-homogeneous $p$-Laplace problems]{Game theoretical asymptotic mean value properties \\ for non-homogeneous $p$-Laplace problems}

\author[F.~del~Teso]{F\'elix del Teso}

\address[F. del Teso]{Departamento de Matemáticas, Universidad Aut\'onoma de Madrid (UAM), Campus de Cantoblanco, 28049 Madrid, Spain} 

\email[]{felix.delteso\@@{}uam.es}

\urladdr{https://sites.google.com/view/felixdelteso}


\author[J. D. Rossi]{Julio D. Rossi}
\address[J. D. Rossi]{Departamento de Matem\'atica, FCEyN, Universidad de Buenos Aires,
 Pabell\'on I, Ciudad Universitaria (1428),
Buenos Aires, Argentina.}
\email[]{jrossi@dm.uba.ar}


\begin{abstract}
We extend the classical mean value property for the Laplacian operator to address a nonlinear and non-homogeneous problem related to the $p$-Laplacian operator for $p>2$. 
Specifically, we characterize viscosity solutions to the $p$-Laplace equation $\Delta_p u \coloneqq \nabla\cdot(|\nabla u|^{p-2} \nabla u) = f$ with a nontrivial right-hand side $f$, through novel asymptotic mean value formulas. While asymptotic mean value formulas for the homogeneous case ($f = 0$) have been previously established, leveraging the normalization $\Delta_p^{\textup{N}} u\coloneqq |\nabla u|^{2-p} \Delta_p u = 0$, which yields the 1-homogeneous normalized $p$-Laplacian, 
such normalization is not applicable when $f \neq 0$. 
Furthermore, the mean value formulas introduced here motivate, for the first time in the literature, a game-theoretical approach for non-homogeneous $p$-Laplace equations. We also analyze the existence, uniqueness, and convergence of the game values, which are solutions to a dynamic programming principle derived from the mean value property.
\end{abstract}

\maketitle

\tableofcontents 

\addtocontents{toc}{\protect\setcounter{tocdepth}{1}}

\section{Introduction}\label{sec:intro}
The aim of this paper is to propose and study game-theoretical asymptotic expansions, asymptotic mean value properties and dynamic programming principles for the classical $p$-Laplacian:
\begin{align*}
    \Delta_p \varphi \coloneqq \nabla \cdot \left(\left|\nabla \varphi\right|^{p-2}\nabla \varphi\right).
\end{align*}

We begin by briefly reviewing the state of the art on this topic, which has been extensively studied over the last decade. The work by Manfredi, Parviainen, and Rossi \cite{MaPaRo10} is the first known result in this direction. They benefit from the identity
\begin{align*}
    \Delta_p\varphi = |\nabla \varphi|^{p-2} \nplap\varphi, \quad \text{with} \quad \nplap\varphi \coloneqq (p-2) \left\langle D^2\varphi \frac{\nabla \varphi}{|\nabla \varphi|}, \frac{\nabla \varphi}{|\nabla \varphi|}\right\rangle + \Delta \varphi,
\end{align*}
where $\nplap$ is the so-called \emph{normalized} $p$-Laplacian. This formulation is particularly useful due to the following fundamental property of $p$-harmonic functions:
\begin{align}\label{eq:equivalencepharmonic}
    \plap u = 0 \quad \text{if and only if} \quad \nplap u = 0.
\end{align}
In this way, they proposed an asymptotic expansion for $p$-harmonic functions via an asymptotic expansion of the normalized $p$-Laplacian:
\begin{theorem}[Manfredi-Parviainen-Rossi 2010]\label{thm:MaPaRo10}
    Let $d \in \mathbb{N}$, $p \geq 2$, $x\in \R^d$, and $\varphi \in C^2(B_R(x))$ for some $R > 0$, such that $\nabla \varphi(x) \neq 0$. For $r > 0$ small enough, define
    \begin{align*}
        \mathcal{M}_{r}[\varphi](x) \coloneqq \beta\left(\frac{1}{2} \sup_{ B_{\gamma r}(x)}\varphi + \frac{1}{2} \inf_{ B_{\gamma r}(x)}\varphi \right) + (1-\beta) \fint_{ B_{\gamma r}(x)} \varphi(y) \, \mathrm{d}y,
    \end{align*}
    where $\beta = (p-2)/(p+d)$ and $\gamma = \sqrt{2(p+d)}$. Then,
    \begin{align*}
        \mathcal{M}_r[\varphi](x) = \varphi(x) + r^2 \nplap \varphi(x) + o(r^2) \quad \text{as} \quad r \to 0^+.
    \end{align*}
    
   As a consequence, a continuous function $u$ satisfies $\Delta_p u (x) = 0$ in the viscosity sense if and only if $ u(x)  = \mathcal{M}_r[u] (x) + o(r^2)$
    in the viscosity sense.
\end{theorem}

Moreover, the above result allows for a game-theoretical interpretation of $p$-harmonic functions through the so-called \emph{tug-of-war with noise},
a two-players zero-sum game whose rules are the following: 
\begin{enumerate}[\noindent\rm (i)]
    \item Fix a parameter $r>0$, an open bounded domain $\Omega$, a starting point $x\in \Omega$ and a payoff function $g$ defined in $\R^d\setminus\Omega$.
    
       \medskip
    
    \item\label{tugogwarnoise-item2} Each turn, toss a biased coin with probabilities $\beta$ for heads and $1-\beta$ for tails.
    \begin{enumerate}[$\bullet$]
    \item If the result is tails, the new position of the game is chosen randomly in the ball of radius
 $\gamma r$.
 \item If the result is heads, a round of tug-of-war is played: a fair coin decides
which of the two players chooses the next position of the game inside the ball of radius
 $\gamma r$.
 \end{enumerate}
 
    \medskip
 
  \item The process described in \eqref{tugogwarnoise-item2} is repeated until the position of the game lies outside $\Omega$ for the first time (this position is denoted by $x_\tau$). When this happens, the second player pays the amount ~$g(x_\tau)$ to the first player.
\end{enumerate}
The value of the game $u_r:\R^d\to\R$ solves the following boundary value problem:
 \begin{align}\label{eq:DPPintrotugogwarwithnoise}
\left\{
\begin{aligned}
    u_r(x)&=   \mathcal{M}_{r}[u_r](x) &\textup{if}&\quad x\in  \Omega,\\
    u_r(x)&=g(x) &\textup{if}& \quad x\in \R^d\setminus\Omega.
\end{aligned}\right.
\end{align}
In the game-theoretical literature, the equation in \eqref{eq:DPPintrotugogwarwithnoise} is known as the Dynamic Programming Principle (DPP) of the game. 
It can be proved (see \cite{MaPaRo10,PSSW,PS}) that the value functions $u_{r}$ converge
uniformly as $r \to 0^+$ to the unique viscosity solution $u$ of $\Delta_p u = 0$ in $\Omega$
with $u=g$ on $\partial\Omega$.

One limitation of the above result is that it is only valid for $p$-harmonic functions, since the equivalence \eqref{eq:equivalencepharmonic} does not hold for nontrivial right-hand sides. Consequently, asymptotic mean value formulas for problems involving $\nplap$ do not yield, in general, asymptotic mean value formulas for problems involving $\plap$. An attempt to overcome this limitation was made by del Teso and Lindgren in \cite{dTLi21} (see also \cite{BuSq22}), where they present a direct asymptotic expansion for the $p$-Laplacian.

\begin{theorem}[del Teso-Lindgren 2021]\label{thm:dTLi21}
    Let $d \in \mathbb{N}$, $p \geq 2$, $x\in \R^d$, and $\varphi \in C^2(B_R(x))$ for some $R > 0$. For $r > 0$ small enough, define
    \begin{align*}
        \mathcal{L}_{r}[\varphi](x) \coloneqq \frac{\kappa_{d,p}}{r^p} \fint_{B_r} |\varphi(x+y) - \varphi(x)|^{p-2} (\varphi(x+y) - \varphi(x)) \, \mathrm{d}y,
    \end{align*}
    where $\kappa_{d,p}  =2(p+d)/(d\fint_{\partial_{B_1}}|y_1|^p \dd \sigma(y))$. Then,
    \begin{align*}
        \mathcal{L}_r[\varphi](x) = \plap \varphi(x) + o_r(1) \quad \text{as} \quad r \to 0^+.
    \end{align*}
    
    As a consequence, a continuous function $u$ satisfies $\Delta_p u (x) = f (x)$ in the viscosity sense if and only if $\mathcal{L}_r[u] (x) = f (x) + o(r^2)$ in the viscosity sense.
\end{theorem}

We observe that \Cref{thm:dTLi21} provides a framework to approximate the $p$-Laplacian and produce approximations in the form of dynamic programming principles (see \cite{dTLi21}). Nevertheless, the inability to express $\mathcal{L}_r[\varphi](x)$ in the form $A_{r}[\varphi](x) - \varphi(x)$, where $A_r$ is a monotone averaging operator, does not allow to provide a probabilistic interpretation of the $p$-Laplacian $\Delta_p$ derived from this asymptotic expansion.

The \textbf{main goal of this paper} is to obtain direct asymptotic expansions for the $p$-Laplacian having the following properties:
\begin{enumerate}[(i)]
\item They allow to characterize solutions of non-homogeneous problems related to the $p$-Laplacian through asymptotic mean value formulas.

\medskip

\item They have an associated Dynamic Programming Principle with a probabilistic game-theoretical interpretation.
\end{enumerate}

In order to explain the underlying ideas of the formulas and properties that will be displayed throughout the paper, let us introduce them in a simplified scenario and operate in a very informal way. 
Consider the problem $\plap u = 1$ for $p = 3$, that is
\begin{align}\label{eq:toyproblem}
\Delta_3 u(x) \coloneqq |\nabla u(x)| \Delta_3^{\textup{N}} u(x) = 1.
\end{align}
Note that we necessarily have $|\nabla u(x)| > 0$ and $\Delta_3^{\textup{N}} u(x) > 0$. It is easy to check (see Lemma \ref{lem:ident-crucial}) that, given $a,b \geq 0$, we have $a^{\frac{1}{2}}b^{\frac{1}{2}} = \inf_{c > 0}\left\{ \frac{c}{2} a + \frac{1}{2c}b \right\}$.
Thus, \Cref{eq:toyproblem} can be equivalently expressed as
\begin{equation}\label{eq:toyproblembis}
    \inf_{c > 0}\left\{ \frac{c}{2} |\nabla u(x)| + \frac{1}{2c}\Delta_3^{\textup{N}} u(x) \right\} = 1.
\end{equation}
Now, we use that for $\rho$ small we have 
$
|\nabla u(x)| \sim \frac{1}{\rho} (\sup_{B_{\rho} (x)}u - u(x))$
and that, from Theorem \ref{thm:MaPaRo10} with $r$ small, we obtain 
$\Delta_3^{\textup{N}} u(x) \sim \frac{1}{r^2} \left(\mathcal{M}_{r}[u](x)-u(x)\right) $
 to get
\begin{align*}
    \inf_{c > 0}\left\{ \frac12 \frac{c}{\rho} \left(\sup_{B_{\rho} (x)}u - u(x) \right)\! + \! \frac{1}{2} \frac{1}{cr^2} \left(\mathcal{M}_{r}[u](x)-u(x)\right)  
    \right\} \sim 1.
\end{align*}
In order to write the left-hand side of this asymptotic formula in the form of 
$A_{\veps}[u](x) - u(x)$, we just choose $\rho$ and $r$ depending on $c$ and a small 
parameter $\varepsilon>0$. More specifically, we take $\rho = c \varepsilon^2$ and $r = \varepsilon c^{-1/2}$ to
formally 
obtain the asymptotic expansion
$$
    A_{\veps}[u](x) \sim u(x) + \varepsilon^2 \quad \textup{with} \quad   A_{\veps}[u](x)\coloneqq\inf_{c > 0}\left\{ \frac{1}{2} \sup_{B_{\varepsilon^2 c}(x)} u + \frac{1}{2} \mathcal{M}_{\varepsilon c^{-1/2}}[u](x) \right\},
$$
and, hence, we arrive to  the desired asymptotic mean value formula
 \begin{equation} \label{0-4}
 A_{\veps}[u](x)= u(x) + \varepsilon^2 +o (\varepsilon^2)
\end{equation}
that formally characterizes solutions to \eqref{eq:toyproblem}.
The associated dynamic programming principle (DPP) for this asymptotic mean value formula just drops the error term and
reads as
\begin{equation} \label{0-5}
A_{\veps}[u](x)= u(x)+ \varepsilon^2.
\end{equation}

Now, let us describe two main properties of 
the asymptotic mean value formula \eqref{0-4} that allow us to
achieve the proposed goals. 
The left-hand side of the mean value formula 
\eqref{0-4} is monotone increasing with respect to $u$
and, in addition, the coefficients add up to 1, so they can play the role of conditional probabilities
 (this fact will be crucial to show that 
there exists a solution to the DPP for the game that we describe in the next section).

\subsection*{Main difficulties.} Let us discuss the main difficulties that arise when we 
want to make this argument rigorous.

\noindent $\bullet$  For a general $p>2$, and assuming that $f> 0$, we write the problem $\Delta_p u (x) = f(x)$ as
\begin{align*}
|\nabla u(x)|^{\frac{p-2}{p-1}} \left(\Delta_p^{\textup{N}} u(x)\right)^{\frac{1}{p-1}} = (f(x))^{\frac{1}{p-1}}.
\end{align*}
We can then reproduce the 
previous idea using the identity $a^{\alpha}b^{1-\alpha}=\inf_{c>0}\left\{\alpha c^{1-\alpha} a + (1-\alpha)c^{-\alpha} b\right\}$ (see \Cref{lem:ident-crucial}),  that is valid for $a,b \geq 0$ and $\alpha\in(0,1)$, by taking $\alpha=(p-2)/(p-1)$. However, when $p\in(1,2)$ we have that $(p-2)/(p-1)<0$ and this methodology does not work.

\noindent $\bullet$ We are performing approximations for $\rho = c \varepsilon^2$ and $r = \varepsilon c^{-1/2}$ small, but in the infimum we consider every $c > 0$. To tackle this,
we restrict $c$ and compute the infimum for $c\in [m(\varepsilon),  M(\varepsilon)]$ with
$m(\varepsilon) \to 0^+$ and $M(\varepsilon) \to + \infty$ as $\veps\to0^+$ (in order to cover the whole set
$(0,+\infty)$ in the limit as $\varepsilon\to 0^+$), and
$\varepsilon (m(\varepsilon))^{-1/2} \to 0^+$ and $\varepsilon^2 M(\varepsilon) \to 0^+$ as $\veps\to0^+$
(to make rigorous the asymptotic expansions). A suitable choice
for $m(\varepsilon)$ and $M(\varepsilon)$ is given by
formula \eqref{as:trunc} below. 
Observe that, when we compute the infimum for $m(\varepsilon) \leq c \leq M(\varepsilon)$ with
$\varepsilon (m(\varepsilon))^{-1/2} \to 0^+$ and $\varepsilon^2 M(\varepsilon) \to 0$,
we have that formulas \eqref{0-4} and \eqref{0-5} are \emph{local}. This holds since, as $\varepsilon \to 0^+$,
the radii of the involved balls where we consider values of $u$ goes to zero uniformly.

\noindent $\bullet$  Moreover, an extra difficulty comes from the fact that we want to
deal with solutions to 
$\Delta_3 u = |\nabla u| \Delta_3^{\textup{N}} u = f$ without assuming a sign condition on $f$. Indeed, the very first step in the previous formal computation writes the product
of two quantities as an infimum, $a^{\frac{1}{2}}b^{\frac{1}{2}} = \inf_{c > 0}\left\{ \frac{c}{2} a + \frac{1}{2c}b \right\}$, but this can not be used when $a$ or $b$ are negative
(the right-hand side is $-\infty$ and the left-hand side does not make sense). 
Hence, to handle this issue we need to introduce two different averaging operators 
$\mathcal{A}_\veps^+ [u]$ and $\mathcal{A}_\veps^- [u]$ (see \Cref{sec:mainpmayor2}).
Heuristically speaking, for a solution to the PDE, $|\nabla u(x)|$ is always nonnegative and $ \Delta_3^{\textup{N}} u(x)$
has the same sign of $f(x)$. Then, when $f(x)<0$ we will use the formula that holds for $-u$ (that is a solution
to the same equation with a change of sign). This change turns infimums into suppremums and viceversa in the previous
formula and give rise to two different mean value formulas
according to the sign of $f(x)$. 

\noindent $\bullet$  On top of these technicalities, solutions to $\Delta_p u = f$ are not in general $C^2$ but $C^{1,\alpha}$ with $\alpha<1$ (\cite{ATU1,ATU2}), and then
we cannot perform second order expansions of the solution. Hence, we need to
interpret the mean value formula and the corresponding PDE in the viscosity sense
(this fact already appeared in previous works \cite{BChMR,dTLi21,MaPaRo10}).

\noindent $\bullet$ The fact that we have two different formulas, $\mathcal{A}_\veps^+[u](x)$ and $\mathcal{A}_\veps^-[u](x)$, depending on the sign of $f(x)$ creates several difficulties when dealing with 
rigorous proofs. The main difficulty appears at points where $f(x)=0$, since at those points
we are forced to choose one of the two formulas, $\mathcal{A}_\veps^+[u](x)$ or $\mathcal{A}_\veps^-[u](x)$. 
The technical problem is tackled restricting ourselves to consider only test functions with non-vanishing gradient
at those points. 

\noindent $\bullet$ In addition, difficulties also arise when we prove the convergence of solutions to the DPP \eqref{0-5}
to solutions to the Dirichlet problem for the $p$-Laplacian. In this case we have to
remark that the DPP \eqref{0-5} is not consistent in the sense of the work \cite{BarSou} by Barles and Souganidis. This fact does not allow us 
to apply their classical  convergence results for viscosity solutions since we need to look carefully which test functions are admissible for our problem.

\noindent $\bullet$ Finally, let us point out that, when we aim to show the equivalence between
being a solution to the PDE $\Delta_p u = f$  and the mean value formula (and also for the convergence of the associated DPP),
we will need some control on the error term in the asymptotic expansion for a 
smooth function at points where $f=0$. 
Hence, we work with $C^3$ test functions, and prove a quantitative version of our previous results. We verify the equivalence
of being a smooth ($C^3$) sub or supersolution to the PDE with being a smooth
sub or supersolution to an asymptotic mean value formula. After this
we obtain the equivalence of being a viscosity solution to the PDE
with verifying an appropriate asymptotic mean value formula also in the viscosity sense.

\subsection*{Description of the main novelties.} The main novelties in this paper can be summarized as follows:

\noindent $\bullet$ For the first time in the literature, we obtain a
mean value formula that characterizes solutions to the non-homogeneous  $p$-Laplace problem and that is well suited to design a game whose value functions verify a Dynamic Programming Principle that is given by the
mean value formula dropping the error term. This task involves an operator that is nonlinear 
and homogeneous of degree $p-1$ (remark that the obtained mean value formula
is homogeneous of degree $1$ in $u$). 
In addition, associated with the
mean value formula, we can design a game whose
value functions converge uniformly to the unique solution to the Dirichlet problem
for the PDE $\Delta_p u = f$. Therefore, here we introduce a new approach to
find game theoretical mean value properties that is flexible enough to handle
operators that are not $1-$homogeneous.

\noindent $\bullet$ Partial differential equations and probability are closely related areas of mathematics. 
This relation goes back to the fact that harmonic functions and martingales both satisfy
mean value properties. 
The asymptotic mean value formulas that we found are of the form 
$$u(x) = \mathcal{A}_\varepsilon [u](x) - \varepsilon^2 J_p (f(x))+o(\varepsilon^2),$$
see \eqref{0-4} and Corollary \ref{teo-f>0-intro-bis} below. Here, $\mathcal{A}_\varepsilon$
is an averaging operator that is monotone increasing in $u$
and, moreover, the coefficients add up to 1 and $J_p$ is the signed $\frac{1}{p-1}$ power, $J_p(\xi):=
\textup{sign}(\xi) (\textup{sign}(\xi) \xi )^{\frac{1}{p-1}}$. Then, we can consider solutions to 
the equation
$$u_\varepsilon(x) = \mathcal{A}_\varepsilon [u_\varepsilon](x) - \varepsilon^2 J_p (f(x))$$
inside a domain $\Omega$ with a  fixed outter datum $u_\varepsilon (x) = g(x)$ for $x\in \mathbb{R}^d \setminus 
\Omega$. This equation is the dynamic programming principle DPP associated to a game
(whose rules can be intuitively deduced from the form of the operator $\mathcal{A}_\varepsilon$).
Here we also decribe this game and show that its value function converges as $\varepsilon \to 0^+$
to the solution to $\Delta_p u(x) = f(x)$ in $\Omega$, with $u (x) = g(x)$ on $\partial 
\Omega$. This is the first probabilistic game-theoretical interpretation for non-homogeneous $p$-Laplace problems.

\noindent $\bullet$ The obtained mean value formulas reproduce 
the diffusive character of the $p$-Laplacian. In fact, the diffusion coefficient in the $p$-Laplacian, $\mbox{div} (|\nabla u|^{p-2} \nabla u)$,
is given by $|\nabla u|^{p-2}$, and then the diffusion is large or small according to the size of the gradient. 
The mean value formula and the dynamic programming principle found here reproduce this property. In fact, if we go back to the formula
that we found for $\Delta_3 u = 1$, \eqref{0-4}, we observe that,
when 
the gradient of the function $u$ is large, we have to choose the constant $c$ 
small in order to make the term $\sup_{B_{\varepsilon^2 c}(x)} u -u(x)$ small (recall that we 
aim for an infimum in $c$ in \eqref{0-4} and in \eqref{0-5}). This forces that the radii of the balls where we
play tug-of-war with noise is large (since it is given by $\varepsilon c^{-1/2}$ with $c$ small). 
That is, we have that, the bigger the gradient is, the larger the balls are when
playing tug-of-war with noise. Conversely, when the gradient of $u$ is small, it will be better to choose 
$c$ large and hence the balls in the diffusive part of the mean value formula will be small.

\subsection*{Comments on related literature}
Now, let us briefly comment on previous results concerning mean value formulas. 
It is well known that there exists a mean value formula that characterizes harmonic functions. Specifically, a function $u$ is harmonic 
(i.e., a solution to $\Delta u= \nabla\cdot(\nabla u) = 0$) 
if and only if $u(x)$ equals the mean value of $u$ inside every 
ball $B_\varepsilon(x)$ centered at $x$. A weaker statement, known as the asymptotic mean value property, 
dates back to \cite{Blaschke,Privaloff}. This property states that $u$ is harmonic if and only if 
\[
u(x) = \fint_{B_\varepsilon(x)} u(y)\, \dd y + o(\varepsilon^2).
\]
Extensions of these ideas to the classical Poisson equation $-\Delta u = f$ yield the formula
\[
u(x) = \fint_{B_\varepsilon(x)} u(y)\, \dd y + \frac{\varepsilon^2}{2(n+2)} f(x) + o(\varepsilon^2).
\]
Mean value formulas for operators other than the Laplacian have also been studied. For instance, \cite{Littman.et.al1963} establishes a mean value theorem for linear divergence form operators with bounded measurable coefficients. A simpler statement, expressed in terms of mean value sets, was obtained in \cite{BH2015,Caffarelli.Roquejoffre.2007}. 
For results concerning mean value properties in the sub-Riemannian setting, see \cite{BoLan}.
An extension to general mean value formulas in non-Euclidean spaces can be found in 
\cite{cur}.

The extension of the linear theory to nonlinear operators is relatively recent. As mentioned earlier, a nonlinear mean value property for $p$-harmonic functions was first introduced in \cite{MaPaRo10,PSSW,PS}. For further references and more general equations, see also \cite{TPSS,KAWOHL2012173,Arroyo,BH2015,LiuSchikorra2015,dTMP,HansHartikainen,Angel.Arroyo.Tesis,I, ArroyoHeinoParviainen2017,DoMaRiSt,Lewicka2017,MiRo24,Gonzalvez2024} and the recently published book \cite{BlancandRossi2019}. 
For mean value properties related to the $p$-Laplacian in the Heisenberg group, refer to \cite{LMan,DoMaRiSt22}, while for the standard variational $p$-Laplacian, see the previously quoted work \cite{dTLi21}. For mean value formulas concerning solutions to more general elliptic problems, including Monge-Ampère and $k$-Hessian equations, we refer to \cite{BChMR,BChMR2}. Finally, this topic has also been addressed in the nonlocal literature (see \cite{BuSq22,dTEnLe22,Lew22,dTMeOc24}).

Concerning the interplay between game theory and partial differential equations there is a large literature
dealing with linear problems. Recently, motivated by \cite{PSSW} where the authors
introduce and analyze the tug-of-war game related to the infinity Laplacian, a number of
papers deal with more general nonlinear equations and different
problems including, for example, the 
obstacle problem. We refer to 
\cite{TPSS,Chandra,dTMP,LeMa,PSSW,PS} and the books \cite{BlancandRossi2019} and \cite{Lew20}. 
We highlight that none of these references deal with the non-homogeneous $p$-Laplace problem.

\medskip

\subsection*{Organization of the paper.} The paper is organized as follows: In Section \ref{sect-main}
we describe and state precisely our main results; in Section \ref{sect-asymp} we analyze asymptotic expansions
for the $p$-Laplacian and prove the characterization of viscosity
solutions to the PDE in terms of the asymptotic mean value formula; in Section \ref{sect-dynamic} we deal
with existence, uniqueness and convergence as $\varepsilon \to 0^+$ of solutions 
to the Dynamic Programming Principle (DPP); in Section \ref{sec:probabilitythings} we analyze the game associated with
our mean value formula and show that the game has a value that coincides with the 
unique solution to the DPP; finally, we include two appendixes, in the first one, Appendix \ref{ApA}, we include proofs
of some arithmetic-geometric identities and in the second one, Appendix \ref{app:viscositydef}, we discuss the definition of being a viscosity solution both to the PDE and to the
asymptotic mean value formula. 

\subsection*{Notations.} Given $p \in (2,\infty)$
and $d\in \mathbb{N}$, we introduce the constants
$$\beta=(p-2)/(p+d), \qquad \gamma=\sqrt{2(p+d)} \qquad \mbox{and} \qquad \alpha=(p-2)/(p-1).$$
We will also use $J_p$ to denote the signed $\frac{1}{p-1}$ power, that is, 
\[
 J_p(\xi)\coloneqq\left\{
 \begin{array}{ll}
 \xi^{\frac{1}{p-1}} \quad &\textup{if} \quad \xi\geq0,\\
 -(-\xi)^{\frac{1}{p-1}} \quad  &\textup{if} \quad \xi\leq0.
 \end{array}
 \right.
 \]

\section{Main results} \label{sect-main}

In this section we include the precise statements of our results.

\subsection{Asymptotic expansions and mean value properties} \label{sec:mainpmayor2}

Our first set of main results regards asymptotic expansions and mean value properties for the $p$-Laplacian in the range $p\in(2,\infty)$. Let us carefully introduce all the ingredients before stating the results. As we have mentioned, to make all the computations of this paper rigorous, we need to reduce the set for the infimum in \eqref{eq:toyproblembis} from $\{c>0 \}$ to a compact set. Let us define the following functions:
\begin{equation}\label{as:trunc}\tag{$\textup{M-m}$}
    \textup{Let $m,M:\R_+ \to \R_+$ be defined by $m(\veps)=\veps^{\frac{2}{2+\alpha}}$ and $M(\veps)=\veps^{-\frac{2}{2-\alpha}}$}.
\end{equation}
Let us also consider the following averaging operators, defined for $\veps>0$, $x\in \R^d$, and a bounded Borel function $\varphi:B_R(x)\to \R$ for some $R>0$ large enough: 
\begin{align}\label{def:A+}
    \mathcal{A}_\veps^+[\varphi](x)\coloneqq \inf_{c\in [m(\veps),M(\veps)]}\left\{\alpha \sup_{B_{\veps^2c^{1-\alpha}}(x)} \varphi + (1-\alpha) \mathcal{M}_{\veps c^{-\frac{\alpha}{2}}}[\varphi](x) \right\} 
\end{align}
and
\begin{align}\label{def:A-}
    \mathcal{A}_\veps^-[\varphi](x)\coloneqq\sup_{c\in [m(\veps),M(\veps)]}\left\{\alpha \inf_{B_{\veps^2c^{1-\alpha}}(x)} \varphi + (1-\alpha) \mathcal{M}_{\veps c^{-\frac{\alpha}{2}}}[\varphi](x) \right\}. 
\end{align}

We are now ready to state our first main result. It provides an asymptotic expansion of the $p$-Laplacian both in an asymptotic form for $C^2$ functions and in a quantitative way for $C^3$ functions. Moreover, while the asymptotic formula may fail for $C^2$ functions when the $p$-Laplacian vanishes
 (due to the lack of control of the error term), this problem is avoided when applied to $C^3$ functions.

\begin{theorem}\label{thm:main-p-greater2}
    Let $d\in \mathbb{N}$, $p>2$, $\veps_0\in (0,1)$, $x\in \R^d$, and $\varphi\in C^2(B_{\veps_0}(x))$ such that $\nabla\varphi(x)\not=0$. Assume \eqref{as:trunc}. The following assertions hold: 
    \begin{enumerate}[\rm (a)]
    \item\label{thm:main-p-greater2-item1}  There exists $\overline{\veps}\in(0,\veps_0)$ such that, for all $\veps\in(0,\overline{\veps})$, the following quantity is well defined:
    \[
 \mathcal{A}_\veps[\varphi](x)\coloneqq\left\{\begin{aligned}
 \mathcal{A}_\veps^+[\varphi](x) \quad \textup{if} \quad \Delta_p\varphi(x)\geq0,\\
  \mathcal{A}_\veps^-[\varphi](x) \quad \textup{if} \quad \Delta_p\varphi(x)<0.
 \end{aligned}\right.
\]
\item\label{thm:main-p-greater2-item2} If $\Delta_p\varphi(x)\not=0$, then we have the following asymptotic expansion of the $p$-Laplacian:
\[
\mathcal{A}_\veps[\varphi](x)=\varphi(x)+\veps^2 J_p(\Delta_p\varphi(x)) +o(\veps^2) \quad \textup{as} \quad \veps\to0^+.
\]
\item\label{thm:main-p-greater2-item3} If, additionally, $\varphi\in C^3(B_{\veps_0}(x))$, then, for all $\veps\in(0,\overline{\veps})$, we have that
\[
\mathcal{A}_\veps[\varphi](x)=\varphi(x)+\veps^2 J_p(\Delta_p\varphi(x)) + \veps^2 E(\varphi,\veps),
\]
with \begin{align*}
|E(\varphi,\veps)|\leq&  K_1 \|D^2\varphi\|_{L^\infty(B_{\overline{\veps}}(x))} \veps^{\frac{2(p-2)}{p}} + K_2\left(\|\nabla\varphi\|_{L^\infty(B_{\overline{\veps}}(x))}+\|D^{3}\varphi\|_{L^\infty(B_{\overline{\veps}}(x)) }+ \frac{|D^2\varphi(x)|^2}{|\nabla\varphi(x)|}\right) \veps^{\frac{2}{3p-4}}.
\end{align*}
where $K_1$ and $K_2$ are positive constants depending only on $p$ and $d$.
    \end{enumerate}
\end{theorem}

\begin{remark}
    We can make the choice $\mathcal{A}_\veps[\varphi](x)=\mathcal{A}_\veps^-[\varphi](x)$ if $\Delta_p\varphi(x)=0$ and the above result remains true.
\end{remark}

 Let us note that the above asymptotic expansion holds only when $\nabla\varphi(x)\not=0$ (since \Cref{thm:MaPaRo10} cannot be used to approximate $\nplap\varphi(x)$). This is due to the fact that  $\nplap\varphi(x)$ is not well defined when $\nabla\varphi(x)=0$.  Even with this limitation,  the above results still provide  asymptotic mean value characterizations of viscosity solutions to the non-homogeneous $p$-Laplace equation when $p>2$. Such result follows directly from the following characterization of smooth sub and supersolutions.

\begin{theorem}\label{thm:subsupersmooth-bis}
      Let $d\in \mathbb{N}$, $p>2$, $x\in \R^d$, $\varphi\in C^3(B_{R}(x))$ for some $R>0$, and $f\in C(B_{R}(x))$ such that $f(x)\geq0$. Additionally, assume that both $\nabla\varphi(x)\not=0$ and $\Delta_p\varphi(x)\not=0$ if $f(x)=0$.  Assume also that $\veps>0$ and that \eqref{as:trunc} holds.
Then, 
\begin{align*}
    \begin{aligned}
        &\Delta_p \varphi(x)\leq f(x) \quad  (\textup{resp. } \Delta_p \varphi(x)\geq f(x) )
    \end{aligned}
\end{align*}
if and only if 
\begin{align*}
    \begin{aligned}
        &\mathcal{A}_\veps^+[\varphi](x)\leq \varphi(x)+\veps^2 J_p(f(x)) +o(\veps^2)   \quad 
        (\textup{resp. }\mathcal{A}^+_\veps[\varphi](x)\geq \varphi(x)+\veps^2 J_p(f(x)) +o(\veps^2) ) \quad \textup{as} \quad \veps\to0^+.
    \end{aligned}
\end{align*}
The same result holds if $f(x)\leq 0$  replacing $\mathcal{A}_\veps^+$ by $\mathcal{A}_\veps^-$.
\end{theorem}

Let us finally fix the notation
 \[
 \overline{\mathcal{A}}_\veps[\varphi;f](x)\coloneqq\left\{\begin{aligned}
 \mathcal{A}_\veps^+[\varphi](x) \quad \textup{if} \quad f(x)\geq0,\\
  \mathcal{A}_\veps^-[\varphi](x) \quad \textup{if} \quad f(x)<0,
 \end{aligned}\right. \quad \textup{and} \quad   \underline{\mathcal{A}}_\veps[\varphi;f](x)\coloneqq\left\{\begin{aligned}
 \mathcal{A}_\veps^+[\varphi](x) \quad \textup{if} \quad f(x)>0,\\
  \mathcal{A}_\veps^-[\varphi](x) \quad \textup{if} \quad f(x)\leq 0.
 \end{aligned}\right.
\]

We are now ready to formulate the second main result: We provide two asymptotic mean value characterizations of viscosity solutions to non-homogeneous $p$-Laplace problems. 

\begin{corollary}[Asymptotic mean value property] \label{teo-f>0-intro-bis} 
       Let $d\in \mathbb{N}$, $p>2$, $\Omega\subset\R^d$ be an open set, and $f\in C(\Omega)$.  Assume also that $\veps>0$ and that \eqref{as:trunc} holds. Let $\mathcal{A}_\veps[\varphi;f]$ be defined by either $\overline{\mathcal{A}}_\veps[\varphi;f]$ or $\underline{\mathcal{A}}_\veps[\varphi;f]$. 
Then, $u$ is a viscosity solution to 
\[
\Delta_p u(x)=f(x) \quad \textup{for} \quad x\in \Omega,
\]
if and only $u$ is a viscosity solution to
\begin{align*}
\begin{aligned}
&u(x) =\mathcal{A}_\veps[u;f](x) 
- \veps^2 J_p(f(x)) +o(\veps^2)  \quad \textup{for} \quad x\in \Omega, \quad \textup{as} \quad \veps\to0^+.
\end{aligned}
\end{align*}
\end{corollary}

The precise definitions of viscosity solutions for both the $p$-Laplace equation and the asymptotic mean value property can be found in \Cref{app:viscositydef}.

Observe that in the previous result we are not assuming any a priori regularity for $u$ besides
continuity. Also notice that we have two different mean value operators,  $\overline{\mathcal{A}}_\veps[\varphi;f]$ and $\underline{\mathcal{A}}_\veps[\varphi;f]$, such that the associated mean value formulas characterize viscosity solutions to 
$\Delta_p u(x)=f(x)$. 

\begin{remark}Let us point out that classical weak solutions for the $p$-Laplacian that are obtained minimizing
the energy 
$$
E(u) = \frac1p \int_\Omega |\nabla u(x)|^p  \dd x - \int_\Omega f(x) u(x)  \dd x
$$
coincide with solutions in the viscosity sense, \cite{Is,JuJu,JuLiMan}. Hence, our main result 
also provides a characterization in terms of an asymptotic mean value formula
for weak solutions. 
\end{remark}

\subsection{Dynamic Programming Principle and game theoretical interpretation}
The previous asymptotic expansion naturally gives a dynamic programming principle for the $p$-Laplace Poisson problem
    \begin{align}\label{eq:ppoisson}
       \left\{
        \begin{aligned}
           \Delta_p u(x)&= f(x) &\textup{if}& \quad x\in \Omega,\\
            u(x)&= g(x) &\textup{if}& \quad x\in \partial\Omega.
        \end{aligned}
        \right.
    \end{align}

We will need the following regularity assumption on the domain $\Omega$:
\begin{equation}\label{as:regudom} 
\tag{$\textup{A}_{\Omega}$}
\begin{split}
\textup{$\Omega\subset \R^d$  is a bounded open domain with the uniform exterior ball condition: there exists }\\ \textup{$R>0$ such that for all $x_0\in \partial \Omega$ there exists $z_0\in \R^d\setminus \Omega$ such that $\overline{B_{R}(z_0)}\cap \overline{\Omega}=\{x_0\}$. }
\end{split}
\end{equation}

In the following result, Theorem \ref{thm:DPPexsandconv}, we present the well-posedness of the associated dynamic programming principle as well as the uniform convergence of its solution to the unique solution of \eqref{eq:ppoisson}.

\begin{theorem} \label{thm:DPPexsandconv}
    Let $d\in \mathbb{N}$, $p>2$, $\Omega$ satisfying \eqref{as:regudom},  $f\in C(\overline{\Omega})$ and $g\in C_{\textup{b}}(\R^d\setminus \Omega)$. Assume \eqref{as:trunc}. For every $\veps>0$, the following dynamic programming principle has a unique bounded Borel solution $u_\veps$:
    \begin{align} \label{DPP}
        \left\{
        \begin{aligned}
            u_\veps(x)&= \mathcal{A}_\veps[u_\veps;f](x) - \veps^2 J_p(f(x)) &\textup{if}& \quad x\in \Omega,\\
            u_\veps(x)&= g(x) &\textup{if}& \quad x\in \R^d\setminus\Omega,
        \end{aligned}
        \right.
    \end{align}
where $\mathcal{A}_\veps[\varphi;f]$ is defined by either $\overline{\mathcal{A}}_\veps[\varphi;f]$ or $\underline{\mathcal{A}}_\veps[\varphi;f]$.

Moreover,
    $$u_\veps\to u  \textup{ as $\veps\to0^+$ uniformly in $\overline{\Omega}$ }$$
    where $u$ is the unique viscosity solution of the $p$-Laplace Poisson problem \eqref{eq:ppoisson}.
\end{theorem} 

\subsection*{The $p$-Laplacian gambling house.} The previous result naturally provides us with a game-theoretical approach to solutions of the $p$-Laplace Poisson problem \eqref{eq:ppoisson}. More precisely, the Dynamic Programming Principle \eqref{DPP} corresponds to the
problem satisfied by the value function of the following two-players, zero-sum game, 
that can be called {\it the $p$-Laplacian gambling house}, whose rules are described below:

\begin{enumerate}[\noindent \rm g(i)]
    \item\label{item1} Fix a parameter $\veps>0$, an open bounded domain $\Omega$, a starting point $x_0\in \Omega$, a payoff function $g$ defined in $\R^d\setminus\Omega$, and a running payoff function $f$ defined in $\Omega$.
    
    \medskip
    
\item\label{item2} Each turn, given the actual position $x\in \Omega$, if $f(x)\geq0$, then, the second player (who wants to minimize the final payoff)
 chooses a constant $c \in [m(\veps),M(\veps)]$. Then they toss a biased coin with probabilities $\alpha$ for heads and $1-\alpha$ for tails. 
   \begin{enumerate}[$\bullet$]
   \item If the result is heads, the first player (who wants to maximize) chooses the next position at any point in the ball $B_{\veps^2c^{1-\alpha}}(x)$.
   \item If the result is tails, they play a round of tug-of-war with noise in the ball $B_{\gamma \veps c^{-\frac{\alpha}{2}}}(x)$ with probabilities $\beta$ and $1-\beta$ (see description in \Cref{sec:intro}).
   \end{enumerate}
   
   \medskip
   
   \item\label{item3} If $f(x)<0$ the roles of the players are reversed. More precisely, the first player chooses $c$ and the second player chooses the next position in the ball $B_{\veps^2c^{1-\alpha}}(x)$ if the result of the first coin toss is heads.
   
   \medskip
   
   \item\label{item4} The process described in \rm{g}\eqref{item2} and \rm{g}\eqref{item3} is repeated, with a running payoff at each movement of amount $-\veps^2 J_p(f(x))$. The game continues until the position of the game lies outside $\Omega$ for the first time (this position is denoted by $x_\tau$). When this happens, the second player pays the amount ~$g(x_\tau)$ to the first player.
\end{enumerate}

    Note that, in this game, the players are somehow ``gambling" by choosing the value of $c$. If one player chooses $c$ small, he lets, with probability $\alpha$, the other player to move freely in a small ball (of radius $\veps^2c^{1-\alpha}$ that becomes small if $c$ is small), and they both play tug-of-war with noise with probability $1-\alpha$ in a large ball
    (of radius $\gamma \veps c^{-\frac{\alpha}{2}}$ that becomes large when $c$ is small). On the other hand, if the player chooses $c$ large, the other player moves freely, with probability $\alpha$, in a large ball, and they both play tug-of-war with noise with probability $1-\alpha$ in a smaller ball. This kind of dynamics allows a much richer set of strategies than the usual tug-of-war with noise.

Also notice that the rules of the game are intuitively reflected in the form of the DPP.
In fact, take a point $x$ with $f(x) >0$, at this point the DPP reads as,
$$
u_\veps(x)=  
\inf_{c\in [m(\veps),M(\veps)]}\left\{\alpha \sup_{B_{\veps^2c^{1-\alpha}}(x)} u_\veps + (1-\alpha) \mathcal{M}_{\veps c^{-\frac{\alpha}{2}}}[u_\veps](x) \right\}  - \veps^2 J_p(f(x)).
$$
Now, one thinks about $u_\veps$ as the expected amount that the players will get starting the game at $x$
(this is the meaning of the value function of the game). This equation says that the value function at $x$
equals the result after one play. 
This is determined by  the infimum among constants \( c \). Recall that, when \( f(x) > 0 \), the player aiming to minimize the payoff chooses \( c \). The calculation of the infimum involves taking the weighted average, where the weights are determined by the probabilities of heads (\( \alpha \)) and tails (\( 1-\alpha \)) in a coin toss. This average includes the expected outcomes of selecting the next position. In this context, the player choosing the position seeks to maximize the outcome, which is where the supremum appears. Subsequently, the game involves a ``tug-of-war'' with noise in balls of radius related to \( c \) and \( \varepsilon \). Finally, the running payoff is subtracted, which in this case is \( \varepsilon^2 J_p(f(x)) \).
A similar interpretation holds at points $x$ with $f(x) <0$
since in this case the DPP is
$$
u_\veps(x)=  
\sup_{c\in [m(\veps),M(\veps)]}\left\{\alpha \inf_{B_{\veps^2c^{1-\alpha}}(x)} u_\veps + (1-\alpha) \mathcal{M}_{\veps c^{-\frac{\alpha}{2}}}[u_\veps](x) \right\}  - \veps^2 J_p(f(x)).
$$

The following result states the connection between this game and the solution of the dynamic programming principle.

\begin{theorem}
    Let the assumptions of \Cref{thm:DPPexsandconv} hold. Consider the game described in {\rm{g}\eqref{item1}-\rm{g}\eqref{item2}-\rm{g}\eqref{item3}-\rm{g}\eqref{item4}}. Then, the game has a value, and it is given by the unique solution $u_\veps$ of the dynamic programming principle \eqref{DPP}.
\end{theorem}

A more rigorous statement of the above result will be given in \Cref{sec:probabilitythings} after we introduce all the probability machinery that is required. 

Notice that this result combined with Theorem \ref{thm:DPPexsandconv} gives that 
the value functions of the game approximate the solution to the $p$-Laplace Poisson problem \eqref{eq:ppoisson}
when the parameter that controls the size of the steps in the game, $\varepsilon$, goes to zero.

 \subsection{The limit cases and possible extensions} Here we collect possible extensions of our results and gather some final comments.  Here, for simplicity, we assume $f>0$.

 \noindent $\bullet$ {\bf The limit case $p=+\infty$.} For $p\to+\infty$, we have $ \alpha \to 1$ and $(f(x))^{\frac{1}{p-1}} \to1$, and then the mean value formula reads as
 $$
 \sup_{B_{\veps^2}(x)} u = u(x)  + \varepsilon^2 +o(\varepsilon^2). 
 $$
The corresponding DPP is given by
 \begin{align*}
         \left\{
        \begin{aligned}
             u_\varepsilon (x) & =  \sup_{B_{\veps^2}(x)} u_\varepsilon - \varepsilon^2 
 &\textup{if}& \quad x\in \Omega,\\
            u_\veps(x)&= g(x) &\textup{if}& \quad x\in \R^d\setminus\Omega,
        \end{aligned}
        \right.
    \end{align*}
whose solutions converge to solutions to the Eikonal equation      
 \begin{align} \label{eq:ppoisson-Eik}
        \left\{
        \begin{aligned}
           |\nabla u(x) |&= 1 &\textup{if}& \quad x\in \Omega,\\
            u(x)&= g(x) &\textup{if}& \quad x\in \partial\Omega.
        \end{aligned}
        \right.
    \end{align}

Notice that the limit as $p\to +\infty$ for solutions to the $p$-Laplace Poisson problem 
\eqref{eq:ppoisson} is the solution to \eqref{eq:ppoisson-Eik}, see \cite{BBM}. Therefore, taking $p=+\infty$
in our mean value formulas (or in the corresponding dynamic programming principles) is
compatible with taking the limit as $p \to + \infty$ in the $p$-Laplace Poisson problem.

 \noindent $\bullet$ {\bf The limit case $p=2$.}  For $p=2$, we get $\beta=0$, $\alpha=0$ and $\gamma=\sqrt{2(2+d)}$. Hence, as $(f(x))^{\frac{1}{p-1}}= f(x)$,
 the mean value formula is given by 
 $$
  \fint_{ B_{\gamma \varepsilon}(x)} u (y) \, \mathrm{d}y =    u(x)
 + \varepsilon^2 f(x) +o(\varepsilon^2),
 $$
 that is a well-known formula for solutions to $\Delta u(x) = f(x)$, see \cite{Blaschke,Privaloff}.
The corresponding DPP is given by
 \begin{align*} \label{DPP-infty-2}
        \left\{
        \begin{aligned}
             u_\varepsilon (x) & =   \fint_{ B_{\gamma\varepsilon}(x)} u_\varepsilon (y) \, \mathrm{d}y 
 - \varepsilon^2 f(x)  
 &\textup{if}& \quad x\in \Omega,\\
            u_\veps(x)&= g(x) &\textup{if}& \quad x\in \R^d\setminus\Omega,
        \end{aligned}
        \right.
    \end{align*}
whose solutions converge to solutions to the Poisson problem for the usual Laplacian. 
Hence, also in this case, our asymptotic expensions are compatible with the classical ones. 

 \noindent $\bullet$ {\bf Mean value properties for more general equations.} Using similar techniques as the ones used to prove our results, we can provide asymptotic expansions, asymptotic mean value characterizations of viscosity solutions, dynamic programming principles, and game theoretical interpretations for a large class of related equations. 
 
In general, if one has an asymptotic mean value formula of the form
$$
u(x) = \mathcal{B}_r[u] (x) + o(r^2)
$$
that characterizes solutions to $F(D^2u(x), \nabla u(x), x) =0$,
one can expect that
$$
u(x) = \inf_{c\in [m(\veps),M(\veps)]}
\left\{\alpha \sup_{B_{\veps^2c^{1-\alpha}}(x)} u + (1-\alpha) \mathcal{B}_{\veps c^{-\frac{\alpha}{2}}}[u](x) \right\} - \varepsilon^2 f(x) + o(\varepsilon^2)
$$
charaterizes solutions to
$$
|\nabla u(x)|^{\alpha } (F(D^2u(x), \nabla u(x), x))^{1-\alpha} = f(x),
$$
assuming that $f>0$ (and a similar formula changing inf by sup and modifying $\mathcal{B}$
holds when $f<0$). Nonlocal versions of $F$ can also be treated similarly. For a reference on this kind of problems we refer to \cite{IS13}.

In particular, for $a_1,a_2\geq0$ and $a_3>0$, we can deal with equations of the form
 \begin{equation*}
 |\nabla u (x) |^{a_1} \left(a_2 \left\langle D^2u (x)  \frac{\nabla u}{|\nabla u|} (x) , \frac{\nabla u}{|\nabla u|} (x) \right\rangle  + a_3\Delta u (x) \right)=f(x),
 \end{equation*}
and also with Pucci type operators
 \begin{equation*}
 |\nabla u (x) |^{a_1} \left( \inf_{\lambda I \leq A \leq \Lambda I} \mbox{trace} \left( A \,  D^2 u (x) \right) \right)=f(x),
 \end{equation*}
 as long as $f(x)>0$.

\section{Asymptotic expansions and mean value properties for the $p$-Laplacian} \label{sect-asymp}

The goal of this section is to prove the asymptotic expansion and mean value property results for the $p$-Laplacian in the range $p\in(2,\infty)$, as stated in \Cref{thm:main-p-greater2}, \Cref{thm:subsupersmooth-bis} and \Cref{teo-f>0-intro-bis}.

\subsection{Asymptotic expansions for $C^2$ functions}

Let us start by recalling the result of Theorem \ref{thm:MaPaRo10} related to the average operator 
 \begin{align*}
        \mathcal{M}_{r}[\varphi](x) \coloneqq \beta\left(\frac{1}{2} \sup_{ B_{\gamma r}(x)}\varphi + \frac{1}{2} \inf_{ B_{\gamma r}(x)}\varphi \right) + (1-\beta) \fint_{ B_{\gamma r}(x)} \varphi(y) \, \mathrm{d}y,
    \end{align*}
where $\beta = (p-2)/(p+d)$ and $\gamma =\sqrt{2(p+d)}$. It states that for $p>2$ and $\varphi \in C^2(B_R(x))$ for some $R>0$ and such that $\nabla\varphi(x)\not=0$, we have the following asymptotic expansion for the normalized $p$-Laplacian:
\begin{equation}\label{eq:plapnasexp_littleo}
\mathcal{M}_r[\varphi](x)-\varphi(x) =  r^2\Delta_p^{\textup{N}} \varphi(x) + o(r^2) \quad \textup{as} \quad r\to0^+.
\end{equation} 
Additionally, we recall that the following asymptotic expansion for the modulus of the gradient holds:
\begin{equation}\label{eq:gradasexpsup_littleo}
\sup_{B_r(x)} \varphi-\varphi(x)  = r\left|\nabla \varphi(x)\right| + o(r)  \quad \textup{as} \quad r\to0^+.
\end{equation}
We also introduce the following assumption that can be seen as a weakened version of \eqref{as:trunc}:
\begin{equation}\label{as:trunc2}\tag{$\textup{H}_{T}$}
    \left\{\begin{aligned}
    &\textup{Let } m,M:\R_+ \to \R_+ \textup{ be non-decreasing and non-increasing functions respectively, and such}\\
&\textup{that }m(\veps)\to0^+, \quad M(\veps)\to+\infty, \quad \veps m(\veps)^{-\frac{\alpha}{2}}\to0^+ \quad \textup{and} \quad \veps^2M(\veps)^{1-\alpha}\to0^+ \quad \textup{as} \quad \veps\to0^+.
    \end{aligned}\right.
\end{equation}
\begin{remark}
    It is standard to check that, if \eqref{as:trunc} holds, then \eqref{as:trunc2} holds.
\end{remark}
Finally, we recall the definitions of $\mathcal{A}^+_\veps$ and $\mathcal{A}^-_\veps$ given in \Cref{sec:mainpmayor2},
\begin{align}\label{def:A+bis}
    \mathcal{A}_\veps^+[\varphi](x)\coloneqq\inf_{c\in [m(\veps),M(\veps)]}\left\{\alpha \sup_{B_{\veps^2c^{1-\alpha}}(x)} \varphi + (1-\alpha) \mathcal{M}_{\veps c^{-\frac{\alpha}{2}}}[\varphi](x) \right\} 
\end{align}
and
\begin{align}\label{def:A-bis}
    \mathcal{A}_\veps^-[\varphi](x)\coloneqq\sup_{c\in [m(\veps),M(\veps)]}\left\{\alpha \inf_{B_{\veps^2c^{1-\alpha}}(x)} \varphi + (1-\alpha) \mathcal{M}_{\veps c^{-\frac{\alpha}{2}}}[\varphi](x) \right\}. 
\end{align}

Let us start by proving that  \eqref{def:A+bis} and \eqref{def:A-bis} are well defined under very mild assumptions. This will actually prove a more general version of Theorem \ref{thm:main-p-greater2}\eqref{thm:main-p-greater2-item1}.

\begin{lemma}
    Let $d\in \mathbb{N}$, $p>2$, $\veps_0\in(0,1)$, $x\in \R^d$ and $\varphi:B_{\veps_0}(x)\to\R$ be a bounded Borel function. Assume \eqref{as:trunc2}.
 There exists $\overline{\veps}\in(0,\veps_0)$ such that, for all $\veps\in(0,\overline{\veps})$, the quantities \eqref{def:A+bis} and \eqref{def:A-bis} are well defined.
\end{lemma}
\begin{proof}
Since $c\in [m(\veps),M(\veps)]$, then, for $\veps>0$ small enough, we have that 
\[
B_{\veps^2 c^{1-\alpha}}(x) \subset B_{\veps^2 M(\veps)^{1-\alpha}}(x)\subset B_{\veps_0}(x) \quad \textup{and} \quad B_{\gamma \veps c^{-\frac{\alpha}{2}}}(x) \subset B_{\gamma \veps m(\veps)^{-\frac{\alpha}{2}}}(x)\subset B_{\veps_0}(x),
\]
where the last inclusions follow from assumption \eqref{as:trunc2}. Since the function $\varphi$ is bounded in $B_{\veps_0}$, then all the $\inf$ and $\sup$ involved in \eqref{def:A+bis} and \eqref{def:A-bis} are well-defined. Finally, since $\varphi$ is additionally a Borel function, then the integrals in \eqref{def:A+bis} and \eqref{def:A-bis} are also well-defined.
\end{proof}

We are now ready prove Theorem \ref{thm:main-p-greater2}\eqref{thm:main-p-greater2-item2}. We restate it here for convenience.

\begin{proposition}\label{pro:asexp-greater2} Let $d\in \mathbb{N}$, $p>2$, $\veps_0\in(0,1)$, $x\in \R^d$, and $\varphi\in C^2(B_{\veps_0}(x))$ such that $\nabla\varphi(x)\not=0$ and $\Delta_p\varphi(x)\not=0$. Assume \eqref{as:trunc2} and define, for  $\veps<\veps_0$ small enough, the average
     \[
 \mathcal{A}_\veps[\varphi](x)\coloneqq\left\{\begin{aligned}
 \mathcal{A}_\veps^+[\varphi](x) \quad \textup{if} \quad \Delta_p\varphi(x)\geq 0,\\
  \mathcal{A}_\veps^-[\varphi](x) \quad \textup{if} \quad \Delta_p\varphi(x)<0.
 \end{aligned}\right.
\]
Then,
\begin{align}\label{eq:asexp-proofs}
\mathcal{A}_\veps[\varphi](x)=\varphi(x)+\veps^2 J_p(\Delta_p\varphi(x)) +o(\veps^2) \quad \textup{as} \quad \veps\to0^+.
\end{align}
\end{proposition}
\begin{proof}
    \textbf{Step 1: }Let us begin by assuming that $\Delta_p\varphi(x)>0$. In this case, $\mathcal{A}_\veps[\varphi](x)=\mathcal{A}_\veps^+[\varphi](x)$ and $J_p(\Delta_p\varphi(x))=(\Delta_p\varphi(x))^\frac{1}{p-1}=|\nabla\varphi(x)|^{\alpha}(\nplap\varphi(x))^{1-\alpha}$. Note also that $|\nabla\varphi(x)|>0$ and $\nplap \varphi(x)>0$. Then, by \eqref{eq:plapnasexp_littleo} and \eqref{eq:gradasexpsup_littleo}, we have that
    \begin{align*}\label{eq:asexpid}
 \frac{\mathcal{A}^+_\veps[\varphi](x)-\varphi(x)}{\veps^2}&=\!\!\!\inf_{c\in [m(\veps),M(\veps)]}\left\{\alpha \left(\frac{\displaystyle \sup_{B_{\veps^2c^{1-\alpha}}(x)} \varphi-\varphi(x)}{\veps^2}\right) + (1-\alpha)\left( \frac{\mathcal{M}_{\veps c^{-\frac{\alpha}{2}}}[\varphi](x)-\varphi(x)}{\veps^2} \right)\right\}\\
 &=\!\!\!\inf_{c\in [m(\veps),M(\veps)]}\left\{\alpha \left(c^{1-\alpha}\left|\nabla \varphi(x)\right| + \frac{o(\veps^2c^{1-\alpha})}{\veps^2} \right) + (1-\alpha) \left(c^{-\alpha}\Delta_p^{\textup{N}} \varphi(x) + \frac{o(\veps^2 c^{-\alpha})}{\veps^2}\right)\right\}\\
  &= \!\!\!\inf_{c\in [m(\veps),M(\veps)]}\left\{\alpha c^{1-\alpha} \left(\left|\nabla \varphi(x)\right| + \frac{o(\veps^2c^{1-\alpha})}{\veps^2c^{1-\alpha}} \right) + (1-\alpha) c^{-\alpha} \left( \Delta_p^{\textup{N}} \varphi(x) + \frac{o(\veps^2 c^{-\alpha})}{\veps^2 c^{-\alpha}}\right)\right\}.
\end{align*}
Now let us fix $\delta>0$ small enough such that $|\nabla\varphi(x)|-\delta>0$ and $\Delta_p^{\textup{N}} \varphi(x)-\delta >0$. By \eqref{as:trunc2} we can also choose $\hat{\veps}$ small enough in such a way that $\left|\frac{o(\veps^2c^{1-\alpha})}{\veps^2c^{1-\alpha}}\right|\leq \delta$ and $\left|\frac{o(\veps^2 c^{-\alpha})}{\veps^2 c^{-\alpha}}\right|\leq \delta$ for all $c\in[m(\veps),M(\veps)]$ and all $\veps<\hat{\veps}$. Thus, by \eqref{eq:gm-am}, the above identity implies
\begin{align*}
\frac{\mathcal{A}^+_\veps[\varphi](x)-\varphi(x)}{\veps^2}&\geq  \inf_{c\in [m(\veps),M(\veps)]}\Big\{\alpha c^{1-\alpha} \left(\left|\nabla \varphi(x)\right| -\delta \Big) + (1-\alpha) c^{-\alpha} \Big( \Delta_p^{\textup{N}} \varphi(x) -\delta \right)\Big\}\\
&\geq\inf_{c>0}\Big\{\alpha c^{1-\alpha} \left(\left|\nabla \varphi(x)\right| -\delta \right) + (1-\alpha) c^{-\alpha} \left( \Delta_p^{\textup{N}} \varphi(x) -\delta \right)\Big\}\\
&= \left(\left|\nabla \varphi(x)\right| -\delta \right)^{\alpha} \left( \Delta_p^{\textup{N}} \varphi(x) -\delta \right)^{1-\alpha}\\
&=  \left|\nabla \varphi(x) \right|^{\alpha} \left( \Delta_p^{\textup{N}} \varphi(x) \right)^{1-\alpha} + o_\delta(1).
\end{align*}
Since $\delta>0$ was arbitrary, we have proved one of the inequalities of identity \eqref{eq:asexp-proofs}. To prove the reverse inequality, we proceed in a similar way to obtain 
\begin{align*}
\frac{\mathcal{A}^+_\veps[\varphi](x)-\varphi(x)}{\veps^2}&\leq  \inf_{c\in [m(\veps),M(\veps)]}\Big\{\alpha c^{1-\alpha} \left(\left|\nabla \varphi(x)\right| +\delta \right) + (1-\alpha) c^{-\alpha} \left( \Delta_p^{\textup{N}} \varphi(x) +\delta \right)\Big\}.
\end{align*}
We can now use \Cref{lem:gm-am-aprox} together with \eqref{as:trunc2} to get
\begin{align*}
\frac{\mathcal{A}^+_\veps[\varphi](x)-\varphi(x)}{\veps^2}\leq& \left(\left|\nabla \varphi(x)\right| +\delta \right)^{\alpha} \left( \Delta_p^{\textup{N}} \varphi(x) +\delta \right)^{1-\alpha}\\
&+ \alpha \left(\left|\nabla \varphi(x)\right| +\delta \right) m(\veps)^{1-\alpha} + (1-\alpha) \left( \Delta_p^{\textup{N}} \varphi(x) +\delta \right) M(\veps)^{-\alpha}\\
=&  \left|\nabla \varphi(x) \right|^{\alpha} \left( \Delta_p^{\textup{N}} \varphi(x) \right)^{1-\alpha} + o_\delta(1) + o_\veps(1).
\end{align*}
The arbitrariness of $\delta>0$ completes the proof of the remaining inequality in \eqref{eq:asexp-proofs}. With this, we have completed the proof when $\Delta_p\varphi(x)>0$. 

\textbf{Step 2: }Let us assume now that $\Delta_p\varphi(x)<0$. We define $\psi=-\varphi$ and note that $$\Delta_p\psi(x)=-\Delta_p\varphi(x)>0.$$ Thus, we can apply the previous result to get
\begin{align*}
\mathcal{A}_\veps^+[\psi](x)=\psi(x)+\veps^2 J_p(\Delta_p\psi(x))+ o(\veps^2).
\end{align*}
Let us rewrite the above identity in terms of $\varphi$. On one hand,
\[
J_p(\Delta_p\psi(x))=J_p(-\Delta_p\varphi(x))=-J_p(\Delta_p\varphi(x)).
\]
On the other hand
\begin{align*}
    \mathcal{A}^+_\veps[\psi](x)&= \mathcal{A}^+_\veps[-\varphi](x) = \inf_{c\in [m(\veps),M(\veps)]}\left\{\alpha \sup_{B_{\veps^2c^{1-\alpha}}(x)} \{-\varphi\} + (1-\alpha) \mathcal{M}_{\veps c^{-\frac{\alpha}{2}}}[-\varphi](x) \right\}\\
    &= \inf_{c\in [m(\veps),M(\veps)]}\left\{-\alpha \inf_{B_{\veps^2c^{1-\alpha}}(x)} \varphi - (1-\alpha) \mathcal{M}_{\veps c^{-\frac{\alpha}{2}}}[\varphi](x) \right\}\\
    &=-\sup_{c\in [m(\veps),M(\veps)]}\left\{\alpha \inf_{B_{\veps^2c^{1-\alpha}}(x)} \varphi + (1-\alpha) \mathcal{M}_{\veps c^{-\frac{\alpha}{2}}}[\varphi](x) \right\}=- \mathcal{A}^-_\veps[\varphi](x),
\end{align*}
that is,
\begin{align*}
-\mathcal{A}_\veps^-[\varphi](x)=-\varphi(x)-\veps^2 J_p(\Delta_p\varphi(x))+ o(\veps^2),
\end{align*}
which is precisely what we wanted to prove.
\end{proof}

\subsection{Asymptotic expansions with quantitative error estimates for $C^3$ functions} We rely now on the following quantitative version of Theorem \ref{thm:MaPaRo10} that can be found in Exercise 3.7 in \cite{Lew20}. 

 \begin{lemma}\label{lem:asexp-plapnorm}
 Let $d \in \mathbb{N}$, $p \geq 2$, $x\in \R^d$, and $\varphi \in C^3(B_R(x))$ for some $R > 0$, such that $\nabla \varphi(x) \neq 0$.
Then, there exists $\overline{R}$ small enough such that for all $r\in(0,\overline{R})$, we have that
\begin{align*}
    \mathcal{M}_r[\varphi](x) - \varphi(x)= r^2\Delta_p^{\textup{N}} \varphi(x) + r^2E_1(\varphi,r),
\end{align*}
with
\[
|E_1(\varphi,r)|\leq K \left(\|D^{3}\varphi\|_{L^\infty(B_{\overline{R}}(x)) }+ \frac{|D^2\varphi(x)|^2}{|\nabla\varphi(x)|}\right) r,
\]
where $K$ is a positive constant depending only on $p$ and $d$.
  \end{lemma}
We will also use the following quantitative estimate for the approximation of the modulus of the gradient. 
\begin{lemma}\label{lem:asexp-gradmod}
Let $d\in \mathbb{N}$, $R>0$ and $\varphi\in C^3(B_{R}(x))$. Then, for all $r\in(0,R/2)$, we have that
\begin{align*}
   \sup_{B_{r}(x)}\varphi - \varphi(x) = r|\nabla \varphi(x)| + r E_2(\varphi,r)
\end{align*}
with
\[
|E_2(\varphi,r)|\leq K \|D^{2}\varphi\|_{L^\infty(B_{R/2}(x)) } r
\]
where $K$ is a positive constant depending only on $p$ and $d$.
\end{lemma}
The proof of the above result is standard and we omit it. We are now ready to prove \Cref{thm:main-p-greater2}\eqref{thm:main-p-greater2-item3}, that we restate here for convenience.

\begin{proposition}\label{pro:asexp-greater2-quantitative}
Let $d\in \mathbb{N}$, $p>2$, $\veps_0\in(0,1)$, $x\in \R^d$, and $\varphi\in C^3(B_{\veps_0}(x))$ such that $\nabla\varphi(x)\not=0$. Assume \eqref{as:trunc} and define, for  $\veps<\veps_0$ small enough, the average
     \[
 \mathcal{A}_\veps[\varphi](x)\coloneqq\left\{\begin{aligned}
 \mathcal{A}_\veps^+[\varphi](x) \quad \textup{if} \quad \Delta_p\varphi(x)\geq0,\\
  \mathcal{A}_\veps^-[\varphi](x) \quad \textup{if} \quad \Delta_p\varphi(x)<0.
 \end{aligned}\right.
\]
Then, there exists $\overline{\veps}$ small enough such that for all $\veps\in(0,\overline{\veps})$, we have that
\begin{align*}
 \mathcal{A}_\veps[\varphi](x)-\varphi(x)=\veps^2 J_p(\Delta_p\varphi(x)) +\veps^2E(\varphi,\veps),
\end{align*}
with 
\begin{align}\label{eq:quanterror}
|E(\varphi,\veps)|\leq&  K_1 \|D^2\varphi\|_{L^\infty(B_{\overline{\veps}}(x))} \veps^{\frac{2(p-2)}{p}} + K_2\left(\|\nabla\varphi\|_{L^\infty(B_{\overline{\veps}}(x))}+\|D^{3}\varphi\|_{L^\infty(B_{\overline{\veps}}(x)) }+ \frac{|D^2\varphi(x)|^2}{|\nabla\varphi(x)|}\right) \veps^{\frac{2}{3p-4}},
\end{align}
where $K_1$ and $K_2$ are positive constants depending only on $p$ and $d$.
\end{proposition}

\begin{proof}
    We prove only the case $\Delta_p\varphi(x)\geq0$ since the case $\Delta_p\varphi(x)<0$ follows by replacing $\varphi$ by $-\varphi$ as in the proof of \Cref{pro:asexp-greater2}. We observe first that, since $\nabla\varphi(x)\not=0$, then $\Delta_p \varphi(x)= |\nabla \varphi(x)|^{p-2} \Delta_p^{\textup{N}}\varphi(x)$ with $\Delta_p^{\textup{N}}\varphi(x)\geq0$. We can use now \Cref{lem:gm-am-aprox} to get
\begin{align*}
    J_p(\Delta_p\varphi(x))&=   |\nabla \varphi(x)|^{\alpha} (\Delta_p^{\textup{N}}\varphi(x))^{1-\alpha}\\
    &=\inf_{c\in[m(\veps),M(\veps)]}\{\alpha c^{1-\alpha} |\nabla \varphi(x)| + (1-\alpha) c^{-\alpha}\Delta_p^{\textup{N}}\varphi(x)\} + E_3(\varphi,\veps),
\end{align*}
where $E_3(\varphi,\veps)$ can be bounded using the choice $m(\veps)=\veps^{\frac{2}{2+\alpha}}$ and $M(\veps)=\veps^{-\frac{2}{2-\alpha}}$ given by \eqref{as:trunc} in the following way:
\[
|E_3(\varphi,\veps)|\leq |\nabla\varphi(x)|\veps^{\frac{2(1-\alpha)}{2+\alpha}} + \Delta_p^{\textup{N}}\varphi(x)\veps^{\frac{2\alpha}{2-\alpha}}  \leq \|\nabla\varphi\|_{L^\infty(B_{\overline{\veps}}(x))}\veps^{\frac{2}{3p-4}} + \|D^2\varphi\|_{L^\infty(B_{\overline{\veps}}(x))}\veps^{\frac{2(p-2)}{p}} .
\]
We use now Lemma \ref{lem:asexp-plapnorm} and Lemma \ref{lem:asexp-gradmod} to obtain,
\begin{align*}
    (\Delta_p\varphi(x))^{\frac{1}{p-1}}
    =&\inf_{c\in[m(\veps),M(\veps)]}\Bigg\{\alpha c^{1-\alpha} \left( \frac{\displaystyle\sup_{B_{\veps^2c^{1-\alpha}}(x)}\varphi - \varphi(x)}{\veps^2c^{1-\alpha}} - E_2(\varphi,\veps^2c^{1-\alpha})\right) \\
    &\qquad \qquad \qquad \, + (1-\alpha) c^{-\alpha}\left(\frac{\mathcal{M}_{\veps c^{-\frac{\alpha}{2}}}[\varphi](x) - \varphi(x)}{\veps^2c^{-\alpha}}-E_1(\varphi,\veps c^{-\frac{\alpha}{2}}) \right)\Bigg\} + E_3(\varphi,\veps),
\end{align*}
that is,
\begin{align*}
   \veps^2 (\Delta_p\varphi(x))^{\frac{1}{p-1}}
    =&\!\!\!\inf_{c\in[m(\veps),M(\veps)]}\Bigg\{\alpha \sup_{B_{\veps^2c^{1-\alpha}}(x)}\varphi +(1-\alpha) \mathcal{M}_{\veps c^{-\frac{\alpha}{2}}}[\varphi](x) \\
    &\qquad \quad \quad \,- \veps^2 \alpha c^{1-\alpha} E_2(\varphi,\veps^2c^{1-\alpha}) -\veps^2(1-\alpha) c^{-\alpha}E_1(\phi,\veps c^{-\frac{\alpha}{2}}) \Bigg\}-\varphi(x) + \veps^2E_3(\varphi,\veps)\\
    =&\!\!\!\inf_{c\in[m(\veps),M(\veps)]}\Bigg\{\alpha \sup_{B_{\veps^2c^{1-\alpha}}(x)}\varphi +(1-\alpha) \mathcal{M}_{\veps c^{-\frac{\alpha}{2}}}[\varphi](x)\Bigg\}-\varphi(x) + \veps^2E_4(\varphi,\veps)+  \veps^2E_3(\varphi,\veps).
\end{align*}
where
\begin{align*}
    |E_4(\varphi,\veps)| &\leq \!\!\! \sup_{c\in [m(\veps),M(\veps)]}\big\{\alpha c^{1-\alpha} |E_2(\varphi,\veps^2c^{1-\alpha})| +(1-\alpha) c^{-\alpha}|E_1(\phi,\veps c^{-\frac{\alpha}{2}})|\big\}\\
    &\leq C \!\!\! \sup_{c\in [m(\veps),M(\veps)]}\Bigg\{  \|D^{2}\varphi\|_{L^\infty(B_{\overline{\veps}}(x))} \veps^2c^{2(1-\alpha)} +\left(\|D^{3}\varphi\|_{L^\infty(B_{\overline{\veps}}(x)) }+ \frac{|D^2\varphi(x)|^2}{|\nabla\varphi(x)|}\right) \veps c^{-\frac{3\alpha}{2}}\Bigg\}\\
    & \leq C \! \left(\|D^{2}\varphi\|_{L^\infty(B_{\overline{\veps}}(x))} \veps^{2}M(\veps)^{2(1-\alpha)}+ \left(\|D^{3}\varphi\|_{L^\infty(B_{\overline{\veps}}(x)) }+ \frac{|D^2\varphi(x)|^2}{|\nabla\varphi(x)|}\right) \veps m(\veps)^{-\frac{3}{2}\alpha }   \right)\\
    &= C \! \left(\|D^{2}\varphi\|_{L^\infty(B_{\overline{\veps}}(x))} \veps^{\frac{2(p-2)}{p}}+ \left(\|D^{3}\varphi\|_{L^\infty(B_{\overline{\veps}}(x)) }+ \frac{|D^2\varphi(x)|^2}{|\nabla\varphi(x)|}\right) \veps^{\frac{2}{3p-4}}   \right),
\end{align*}
where $C$ is a positive constant depending only on $p$ and $d$. This concludes the proof.
\end{proof}

\subsection{Asymptotic mean value characterization of viscosity 
 solutions} The first goal of this section is to prove the characterization of smooth sub and supersolutions given by \Cref{thm:subsupersmooth-bis}. We only prove it in the case $f(x)\geq0$ since the result for $f(x)\leq 0$ follows by
 replacing $f$ by $-f$ and $\varphi$ by $-\varphi$ as in the proof of \Cref{pro:asexp-greater2}.  Again, we restate the result here for convenience.

\begin{proposition}\label{prop:subsupersmooth}
  Let $d\in \mathbb{N}$, $p>2$, $x\in \R^d$, $\varphi\in C^3(B_{R}(x))$ for some $R>0$, and $f\in C(B_{R}(x))$ such that $f(x)\geq0$. Additionally, assume that both $\nabla\varphi(x)\not=0$ and $\Delta_p\varphi(x)\not=0$ if $f(x)=0$.  Assume also that  $\veps>0$ and that \eqref{as:trunc} holds.
  
Then, 
\begin{eqnarray}
        &\Delta_p \varphi(x)\leq f(x) \label{eq:smoothsuper-plap}\\
        (\textup{resp. } &\Delta_p \varphi(x)\geq f(x) )\label{eq:smoothsub-plap}
\end{eqnarray}
if and only if 
\begin{eqnarray}
        \mathcal{A}_\veps^+[\varphi](x)\leq \varphi(x)+\veps^2 J_p(f(x)) +o(\veps^2) \quad \textup{as} \quad &\veps\to0^+ \label{eq:smoothsuper-MVP} \\
        (\textup{resp. } \mathcal{A}_\veps^+[\varphi](x)\geq \varphi(x)+\veps^2 J_p(f(x)) +o(\veps^2) \quad \textup{as} \quad &\veps\to0^+)  \label{eq:smoothsub-MVP} .
\end{eqnarray}
\end{proposition}

\begin{proof} 
First, let us note that if $\Delta_p\varphi(x)>0$ (and thus, $\nabla\varphi(x)\not=0$), then the equivalence follows directly from \Cref{pro:asexp-greater2}. From now, let us assume that $\Delta_p\varphi(x)\leq0$.

    \textbf{Case 1:} Let us prove the equivalence \eqref{eq:smoothsuper-plap}$\iff$\eqref{eq:smoothsuper-MVP}.

\noindent \textbf{Step 1: Proof of \eqref{eq:smoothsuper-plap} $\implies$ \eqref{eq:smoothsuper-MVP}.} On one hand, if $\nabla\varphi(x)=0$, we have 
   \begin{align*}
 \frac{\mathcal{A}_\veps^+[\varphi](x)- \varphi(x)}{\veps^2}
 &\leq \inf_{c\in [m(\veps),M(\veps)]}\Big\{\alpha c^{1-\alpha} \delta + (1-\alpha)c^{-\alpha} K\Big\}= \delta^\alpha K^{1-\alpha} + o_\veps(1).
 \end{align*}
 Note that \eqref{eq:smoothsuper-MVP} follows since $\delta>0$ is arbitrary small and $f(x)\geq0$. On the other hand, if $\nabla\varphi(x)\not=0$,  then $\nplap \varphi(x)$ is well defined and $\nplap \varphi(x)\leq0$ since $\plap \varphi(x)\leq0$. Consequently, 
\begin{align*}
 \frac{\mathcal{A}_\veps^+[\varphi](x)- \varphi(x)}{\veps^2}
 &\leq \inf_{c\in [m(\veps),M(\veps)]}\Big\{\alpha c^{1-\alpha} K + (1-\alpha)c^{-\alpha} \delta \Big\}= K^\alpha \delta^{1-\alpha} + o_\veps(1),
 \end{align*}
 and \eqref{eq:smoothsuper-MVP} follows. 
 
\noindent \textbf{Step 2: Proof of \eqref{eq:smoothsuper-MVP} $\implies$ \eqref{eq:smoothsuper-plap}.} There is actually nothing to prove here since we are assuming that $\Delta_p \varphi(x)\leq 0$, so we trivially have $\Delta_p \varphi(x)\leq f(x)$ since $f(x)\geq0$.

 \textbf{Case 2:} Let us prove the equivalence \eqref{eq:smoothsub-plap}$\iff$\eqref{eq:smoothsub-MVP}.
 
 \noindent \textbf{Step 1: Proof of \eqref{eq:smoothsub-plap} $\implies$ \eqref{eq:smoothsub-MVP}.}
 If $f(x)>0$, then, by \eqref{eq:smoothsub-plap}, $\plap\varphi(x)\geq f(x)>0$, which is a contradiction with the fact that we only need to check the case $\plap\varphi(x)\leq0$. On the other hand, if $f(x)=0$, then, by assumption, we have that   $\Delta_p\varphi(x)\not=0$, and thus, $\plap\varphi(x)<0$. This is a contradiction with \eqref{eq:smoothsub-plap}, and the proof is completed.

 \noindent \textbf{Step 2: Proof of \eqref{eq:smoothsub-MVP} $\implies$ \eqref{eq:smoothsub-plap}.} We are already assuming that $\plap \varphi(x)\leq 0$. Let us assume first, by contradiction, that $ \plap \varphi(x) <0$ (and thus $\nabla\varphi(x)\not=0$ and $\nplap\varphi(x) <0$). Then,
    \begin{align*}
 \frac{\mathcal{A}_\veps^+[\varphi](x)- \varphi(x)}{\veps^2} &=\!\!\! \inf_{c\in [m(\veps),M(\veps)]}\left\{\alpha \left(\frac{\displaystyle \sup_{B_{\veps^2c^{1-\alpha}}(x)} \varphi-\varphi(x)}{\veps^2}\right) + (1-\alpha)\left( \frac{\mathcal{M}_{\veps c^{-\frac{\alpha}{2}}}[\varphi](x)-\varphi(x)}{\veps^2} \right)\right\}\\
 &= \!\!\!\inf_{c\in [m(\veps),M(\veps)]}\left\{\alpha c^{1-\alpha} \left(\left|\nabla \varphi(x)\right| + \frac{o(\veps^2c^{1-\alpha})}{\veps^2c^{1-\alpha}} \right) + (1-\alpha)c^{-\alpha} \left(\nplap \varphi(x) + \frac{o(\veps^2 c^{-\alpha})}{\veps^2c^{-\alpha}}\right)\right\}\\
 &\leq \!\!\!\inf_{c\in [m(\veps),M(\veps)]}\Big\{\alpha c^{1-\alpha} \left(2\left|\nabla \varphi(x)\right|\right) + (1-\alpha)c^{-\alpha} \left(\nplap \varphi(x)/2 \right)\Big\}\\
 &\leq \alpha m(\veps)^{1-\alpha} \left(2\left|\nabla \varphi(x)\right|\right) + (1-\alpha)m(\veps)^{-\alpha} \left(\nplap \varphi(x)/2 \right) \\
 &<-1.
 \end{align*}
 where the last inequality follows from the fact that $m(\veps)\to0^+$ as $\veps\to0^+$ and $\nplap\varphi(x)<0$. This is clearly a contradiction with \eqref{eq:smoothsub-MVP} since $f(x)\geq0$. Thus, we must have $\Delta_p\varphi(x)=0$. Let us assume now, by contradiction, that $\nabla\varphi(x)=0$. Then, by assumption, we have that  $f(x)>0$, which implies that
   \begin{align*}
 0<J_p(f(x))\leq \frac{\mathcal{A}_\veps^+[\varphi](x)- \varphi(x)}{\veps^2} + o_\veps(1)
 &\leq \inf_{c\in [m(\veps),M(\veps)]}\Big\{\alpha c^{1-\alpha} \delta + (1-\alpha)c^{-\alpha} K\Big\}= \delta^\alpha K^{1-\alpha} + o_\veps(1).
 \end{align*}
 From here, we reach a contradiction by letting $\veps\to0^+$ and using the arbitrariness of $\delta>0$. At this point, we only need to check the case $\nabla\varphi(x)\not=0$ and $\nplap \varphi(x)=0$. In this case, we can use the quantitative asymptotic expansion estimate given in \Cref{pro:asexp-greater2-quantitative}, that works also in the case when $\Delta_p\varphi(x)=0$ with $\nabla\varphi(x)\not=0$, to prove the desired result.
\end{proof}

Let us now comment on the proof of the mean value characterization of viscosity solutions given by \Cref{teo-f>0-intro-bis}. 

\begin{corollary}[Asymptotic mean value property] \label{teo-f>0-sect} 
       Let $d\in \mathbb{N}$, $p>2$, $\Omega\subset\R^d$ be an open set, and $f\in C(\Omega)$.  Assume also that  $\veps>0$ and that \eqref{as:trunc} holds. Let $\mathcal{A}_\veps[\varphi;f]$ be defined by either $\overline{\mathcal{A}}_\veps[\varphi;f]$ or $\underline{\mathcal{A}}_\veps[\varphi;f]$. 
Then, $u$ is a viscosity solution to 
\[
\Delta_p u(x)=f(x) \quad \textup{for} \quad x\in \Omega,
\]
if and only $u$ is a viscosity solution to
\begin{align*}
\begin{aligned}
&u(x) =\mathcal{A}_\veps[u;f](x) 
- \veps^2 J_p(f(x)) +o(\veps^2)  \quad \textup{for} \quad x\in \Omega, \quad \textup{as} \quad \veps\to0^+.
\end{aligned}
\end{align*}
\end{corollary}

\begin{proof}
The proof follows from the previous result, Proposition \ref{prop:subsupersmooth},
using the definitions of solution to both the PDE and the mean value property
stated in the Appendix \ref{app:viscositydef}. In fact, from Proposition \ref{prop:subsupersmooth}, we have that
$u$ is a viscosity supersolution to the PDE $\Delta_p u(x)=f(x)$ in the sense of Definition \ref{def:viscsol-p-Laplace} (every test function that touches $u$ from below verifies
$\Delta_p u(x)\leq f(x)$, notice that we restrict to test functions with non-vanishing gradient at points where $f(x)=0$) if and only if it is a viscosity supersolution to the mean value formula in the sense of Definition \ref{def:asMVPviscosity}. The same
holds for subsolutions proving the desired characterization of viscosity solutions.  
\end{proof}

\section{Dynamic programming principles for the $p$-Laplacian} \label{sect-dynamic}

The aim of this section is to prove \Cref{thm:DPPexsandconv}. Let us recall the framework. We consider the boundary value problem 
\begin{equation}\label{eq:DPP}\tag{DPP}
\left\{
\begin{aligned}
    u_\veps(x)&= \mathcal{A}_\veps[u_\veps;f](x)-\veps^2J_p(f(x)), &\textup{if}&\quad x\in  \Omega,\\
    u_\veps(x)&=g(x), &\textup{if}& \quad x\in \R^d\setminus\Omega,
\end{aligned}\right.
\end{equation}
with
\[
 \mathcal{A}_\veps[u_\veps;f](x):=\left\{\begin{aligned}
 \mathcal{A}_\veps^+[u_\veps](x) \quad \textup{if} \quad f(x)\geq0,\\
  \mathcal{A}_\veps^-[u_\veps](x) \quad \textup{if} \quad f(x)<0.
 \end{aligned}\right.
\]
\begin{remark}
    We will only prove \Cref{thm:DPPexsandconv} with $\mathcal{A}_\veps[u_\veps;f]$ as given above. The proof for \[
 \mathcal{A}_\veps[u_\veps;f](x):=\left\{\begin{aligned}
 \mathcal{A}_\veps^+[u_\veps](x) \quad \textup{if} \quad f(x)>0,\\
  \mathcal{A}_\veps^-[u_\veps](x) \quad \textup{if} \quad f(x)\leq 0,
 \end{aligned}\right.
\]
follows in a similar way.
\end{remark}

\subsection{Existence and uniqueness and properties of solutions}
We first prove a comparison principle that, in particular, implies uniqueness of solutions. 
\begin{lemma}\label{lem:comparisonDPP}
Let $d\in \N$, $p>2$, $\veps\in(0,1)$, $f,f_1,f_2\in C(\overline{\Omega})$ and $g_1,g_2\in C_{\textup{b}}(\R^d\setminus\Omega)$. Assume \eqref{as:trunc2}. Consider two bounded Borel functions $\overline{u},\underline{u} \colon \R^d \to \R$ such that 
\begin{equation*}
\left\{
\begin{aligned}
    \overline{u}(x)&\geq \mathcal{A}_\veps[\overline{u};f](x)-\veps^2J_p(f_1(x)), &\textup{if}&\quad x\in  \Omega,\\
    \overline{u}(x)&\geq g_1(x), &\textup{if}& \quad x\in \R^d\setminus\Omega,
\end{aligned}\right.
\end{equation*}
and
\begin{equation*}
    \left\{
\begin{aligned}
    \underline{u}(x)&\leq \mathcal{A}_\veps[\underline{u};f](x)-\veps^2J_p(f_2(x)), &\textup{if}&\quad x\in  \Omega,\\
    \underline{u}(x)&\leq g_2(x), &\textup{if}& \quad x\in \R^d\setminus\Omega.
\end{aligned}\right.
\end{equation*}
If $f_1\leq f_2$ and $g_1\geq g_2$, then $\overline{u}\geq \underline{u}$ in $\R^d$. In particular, there exists at most one bounded Borel function $u$ satisfying \eqref{eq:DPP}. 
\end{lemma}
\begin{proof}
Since $\overline{u}\geq \underline{u}$ in $\R^d\setminus \Omega$, we only need to show that $\overline{u}\geq \underline{u}$ in $\Omega$. Assume by contradiction that $\overline{u}(x)< \underline{u}(x)$ for some $x\in \Omega$, and let $x_0\in \Omega$ and $S>0$ be such that
\[
S=\underline{u}(x_0)-\overline{u}(x_0)= \sup_{x\in \R^d}\{\underline{u}(x)-\overline{u}(x)\}.
\]
Define $\underline{v}=\underline{u}-S$, so that $\underline{v}(x_0)=\overline{u}(x_0)$ and $\underline{v}\leq \overline{u}$ in $\R^d$. Moreover,
\begin{align*}
\underline{v}(x)&=\underline{u}(x)-S \\ & \leq \mathcal{A}_\veps[\underline{u};f](x)-S-\veps^2J_p(f_2(x)) \\ & = \mathcal{A}_\veps[\underline{u}-S;f](x)-\veps^2J_p(f_2(x))\\
&= \mathcal{A}_\veps[\underline{v};f](x)-\veps^2J_p(f_2(x)),
\end{align*}
that is,
\begin{equation*}
\left\{
\begin{aligned}
    \underline{v}(x)&\leq \mathcal{A}_\veps[\underline{v};f](x)-\veps^2J_p(f_2(x)), &\textup{if}&\quad x\in  \Omega,\\
    \underline{v}(x)&\leq g_2(x)-S, &\textup{if}& \quad x\in \R^d\setminus\Omega.
\end{aligned}\right.
\end{equation*}
In particular, we have that
\begin{align*}
    0&=\underline{v}(x_0)-\overline{u}(x_0) \\ 
    &\leq \mathcal{A}_\veps[\underline{v};f](x_0)-\mathcal{A}_\veps[\overline{u};f](x_0) - \veps^2 (J_p(f_2(x_0))-J_p(f_1(x_0)))\\
    &\leq \mathcal{A}_\veps[\underline{v};f](x_0)-\mathcal{A}_\veps[\overline{u};f](x_0).
\end{align*}
If $f(x_0)\geq0$, we have, from the above identity
\begin{align*}
    0\leq &\mathcal{A}_\veps^+[\underline{v}](x_0)-\mathcal{A}_\veps^+[\overline{u}](x_0)\\
    = &
    \inf_{c\in [m(\veps),M(\veps)]}\left\{\alpha \sup_{B_{\veps^2c^{1-\alpha}}(x_0)} \underline{v} + (1-\alpha) \mathcal{M}_{\veps c^{-\frac{\alpha}{2}}}[\underline{v}](x_0) \right\}\\
    &-\inf_{c\in [m(\veps),M(\veps)]}\left\{\alpha \sup_{B_{\veps^2c^{1-\alpha}}(x_0)} \overline{u} + (1-\alpha) \mathcal{M}_{\veps c^{-\frac{\alpha}{2}}}[\overline{u}](x_0) \right\}\\
    \leq&\sup_{c\in [m(\veps),M(\veps)]}\left\{\alpha \left(\sup_{B_{\veps^2c^{1-\alpha}}(x_0)} \underline{v}-\sup_{B_{\veps^2c^{1-\alpha}}(x_0)} \overline{u}\right)+ (1-\alpha) \left(\mathcal{M}_{\veps c^{-\frac{\alpha}{2}}}[\underline{v}](x_0)-\mathcal{M}_{\veps c^{-\frac{\alpha}{2}}}[\overline{u}](x_0) \right)\right\}\\
    \leq& (1-\alpha)(1-\beta)\sup_{c\in [m(\veps),M(\veps)]}\left\{ \fint_{B_{\gamma \veps c^{-\frac{\alpha}{2}} }(x_0)}(\underline{v}(y)-\overline{u}(y))\dd y\right\}\\
    \leq & \frac{(1-\alpha)(1-\beta)}{|B_{\gamma \veps m(\veps)^{-\frac{\alpha}{2}}}(x_0)|}\int_{B_{\gamma \veps M(\veps)^{-\frac{\alpha}{2}}}(x_0)}(\underline{v}(y)-\overline{u}(y))\dd y.
\end{align*}
This implies that  $\underline{v}= \overline{u}$ on $B_\eta(x_0)$ for some $\eta>0$  that depends only on $\veps$. A similar argument follows if $f(x_0)< 0$. Repeating this process iteratively at new contact points, we get that $\underline{v}(x)= \overline{u}(x)$ for some $x\in \R^d\setminus \Omega$, and thus, we have that
\[
\overline{u}(x)=\underline{v}(x) \leq g(x)-S \leq \overline{u}(x)-S,
\]
which is a contradiction since $S>0$. 

Uniqueness of solutions of \eqref{eq:DPP} follows by taking $f_1=f_2=f$ and $g_1=g_2$.
\end{proof}

Before proceeding to prove existence of solutions, we need to ensure the existence of a subsolution and a supersolution.

\begin{lemma}\label{lem:barrier}
Let $d\in \N$ $p>2$, $\veps\in(0,1)$, $f\in C(\overline{\Omega})$ and $g\in C_{\textup{b}}(\R^d\setminus\Omega)$. Assume \eqref{as:trunc2}. Then there exist bounded Borel functions $\underline{u}$ and $\overline{u}$ such that
\begin{align}\label{eq:expsub}
     \left\{\begin{aligned}&  \underline{u}(x)  \leq \mathcal{A}_\veps[\underline{u};f](x)-\veps^2 J_p(f(x))&\textup{if}& \quad x\in \Omega,\\
    & \underline{u}(x) \leq g(x)&\textup{if}&\quad  x\in\mathbb{R}^d\setminus \Omega,
    \end{aligned}\right.
\end{align}
and 
\begin{align}\label{eq:expsuper}
     \left\{\begin{aligned}&  \overline{u}(x)  \geq \mathcal{A}_\veps[\overline{u};f](x)-\veps^2 J_p(f(x))&\textup{if}& \quad x\in \Omega,\\
    & \overline{u}(x) \leq g(x)&\textup{if}&\quad  x\in\mathbb{R}^d\setminus \Omega.
    \end{aligned}\right.
\end{align}
\end{lemma}
\begin{proof}
    First let us observe, that given any function $\underline{u}$, we have
    \[
    \mathcal{A}_\veps[\underline{u};f](x) \geq  \inf_{c\in [m(\veps),M(\veps)]}\left\{\alpha \inf_{B_{\veps^2c^{1-\alpha}}(x)} \underline{u} + (1-\alpha) \mathcal{M}_{\veps c^{-\frac{\alpha}{2}}}[\underline{u}](x) \right\}.
    \]
Since $\Omega$ is a bounded set, we can choose $R>0$ large enough such that $\Omega \subset B_R(0)$ and 
\[
(B_{\veps^2 M(\veps)^{1-\alpha}}(0)+ \Omega)\cup (B_{ \veps m(\veps)^{-\frac{\alpha}{2}}\sqrt{2(p+d)}}(0)+ \Omega) \subset B_R(0).
\]
Given two generic constants $L,T>0$, to be chosen later, and using the convention $x=(x_1,\ldots,x_d)$, let us define
\begin{align*}
    \underline{u}(x)=\left\{
    \begin{aligned}
    &e^{L x_1}-T  &\textup{if}& \quad x\in B_R(0),\\
    &-\|g\|_{L^\infty(\R^d\setminus\Omega)} \quad &\textup{if}& \quad  x\in \mathbb{R}^d\setminus B_R(0),
    \end{aligned}
    \right.
\end{align*}
which is clearly a bounded Borel function. Note that, in $B_R(0)$, we have that $\underline{u}$ and all its derivatives are strictly increasing functions in its first variable $x_1$. More precisely, for all $n\in \mathbb{N}$ and all $x\in B_R(0)$, we have
\[
\partial_{x_1}^n \underline{u}(x) = L^n e^{L x_1} >0.
\]
Moreover, for every $x\in \Omega$, we have
\[
\inf_{B_{\veps^2c^{1-\alpha}}(x)} \underline{u}-\underline{u}(x)= -\veps^2 c^{1-\alpha} \partial_{x_1} \underline{u}(x) + \frac{(\veps^2 c^{1-\alpha})^2 }{2} \partial_{x_1}^2 \underline{u}(\xi)\geq -\veps^2 c^{1-\alpha} L e^{Lx_1},
\]
where $\xi\in (x_1-\veps^2c^{1-\alpha}, x_1)$. In a similar way,
\begin{align*}
    \frac{1}{2} \inf_{B_{\gamma \veps c^{-\frac{\alpha}{2}}}(x)}\underline{u}+  \frac{1}{2} \sup_{B_{\gamma\veps c^{-\frac{\alpha}{2}}}(x)}\underline{u}-2\underline{u}(x) \geq \veps^2c^{-\alpha}\partial^2_{x_1}\underline{u}(x) = \veps^2c^{-\alpha} L^2 e^{L x_1},
\end{align*}
and
\[
\fint_{B_{\gamma \veps c^{-\frac{\alpha}{2}}}(x)} \underline{u}(y) \dd  y - \underline{u}(x) \geq  \veps^2c^{-\alpha} K_{d} L^2 e^{Lx_1},
\]
where $K_d$ is a positive constant depending only on the dimension. Thus, there exists a constant $\tilde{K}_d$, depending only on the dimension, such that 
\[
\mathcal{M}_{\veps c^{-\frac{\alpha}{2}}}[\underline{u}](x)-\underline{u}(x) \geq \veps^2c^{-\alpha} \tilde{K}_{d} L^2 e^{Lx_1}.
\]

We are now ready to show that $\underline{u}$ satisfies \eqref{eq:expsub} for a suitable choice of the constants $L$ and $T$. By direct computations,
\begin{align*}
    &\frac{ \mathcal{A}_\veps[\underline{u};f](x)-\underline{u}(x)}{\veps^2} -J_p(f(x)) \\
    &\geq \inf_{c\in [m(\veps),M(\veps)]}\left\{\alpha \frac{\displaystyle \inf_{B_{\veps^2c^{1-\alpha}}(x)} \underline{u}-\underline{u}(x)}{\veps^2} + (1-\alpha) \frac{\mathcal{M}_{\veps c^{-\frac{\alpha}{2}}}[\underline{u}](x)-\underline{u}(x)}{\veps^2} \right\}- \|f\|_{L^\infty(\Omega)}^{\frac{1}{p-1}}\\
    &\geq e^{Lx_1} \inf_{c\in [m(\veps),M(\veps)]}\left\{\alpha c^{1-\alpha} (-L) + (1-\alpha) c^{-\alpha}\left(\tilde{K}_d L^2\right) \right\}- \|f\|_{L^\infty(\Omega)}^{\frac{1}{p-1}}\\
    &\geq e^{Lx_1} \left( - \alpha M(\veps)^{1-\alpha} L + (1-\alpha) M(\veps)^{-\alpha}\tilde{K}_d L^2- \|f\|_{L^\infty(\Omega)}^{\frac{1}{p-1}}\right)\\
    &\geq0
\end{align*}
where the last estimate follows by taking $L$ big enough, since the positive term depends quadratically on $L$ while the negative ones depend sub-quadratically on $L$. Thus, for all $x\in \Omega$, we have that
\[
\underline{u}(x) \leq\mathcal{A}_\veps[\underline{u};f](x) - \veps^2 J_p(f(x)),
\]
and this holds independently on the choice of $T$. Finally, we can choose $T$ large enough such that
\[
\underline{u}(x)  \leq -\|g\|_{L^\infty(\R^d\setminus\Omega)} 
\]
for all $x\in B_R(0)$. This implies that $\underline{u}\leq g$ in $\R^d\setminus \Omega$, which concludes the proof of the fact that $\underline{u}$ satisfies \eqref{eq:expsub}.

To prove the existence of $\overline{u}$ satisfying \eqref{eq:expsuper} it is enough to take $\overline{u}=-\underline{u}$ and observe that
\[
 \mathcal{A}_\veps[\overline{u};f](x) \leq  \sup_{c\in [m(\veps),M(\veps)]}\left\{\alpha \sup_{B_{\veps^2c^{1-\alpha}}(x)} \overline{u} + (1-\alpha) \mathcal{M}_{\veps c^{-\frac{\alpha}{2}}}[\overline{u}](x) \right\}.
\]
\end{proof}
We are now ready to establish existence of solutions of the dynamic programming principle \eqref{eq:DPP}.
\begin{lemma}
Let $d\in \N$, $p>2$, $\veps\in(0,1)$, $f\in C(\overline{\Omega})$ and $g\in C_{\textup{b}}(\R^d\setminus\Omega)$. Assume \eqref{as:trunc2}. Then there exists a unique bounded Borel function $u_\veps$ satisfying \eqref{eq:DPP}. 
\end{lemma}

\begin{proof}
Uniqueness is already established. Let us prove existence. By Lemma \ref{lem:barrier}, we can find $u_0$ to be a bounded Borel function such that
\begin{align*}
     \left\{\begin{aligned}&  u_{0}(x)  \leq \mathcal{A}_\veps[u_{0};f](x)-\veps^2 J_p(f(x))&\textup{if}& \quad x\in \Omega,\\
    & u_{0}(x) \leq g(x)&\textup{if}&\quad  x\in\mathbb{R}^d\setminus \Omega,
    \end{aligned}\right.
\end{align*}
and define $\{u_k\}_{k\in \mathbb{N}}$ by recursion as
\begin{align*}
    u_{k}(x)= \left\{\begin{aligned}&\mathcal{A}_\veps[u_{k-1};f](x)-\veps^2 J_p(f(x))&\textup{if}& \quad x\in \Omega,\\
    &g(x)&\textup{if}&\quad  x\in\mathbb{R}^d\setminus \Omega.
    \end{aligned}\right.
\end{align*}

\noindent\textbf{Step 1:} Let us prove that, for every $x\in \R^d$, the sequence $\{u_k(x)\}_{k\in \mathbb{N}}$ is a nondecreasing sequence in $k$.
Clearly, $u_0\leq u_1$ in $\mathbb{R}^d$. Assume, using an inductive argument, that $u_{k-1}\leq u_k$ in $\mathbb{R}^d$. Then, on one hand, if $x\in \R^d\setminus \Omega$, we have
\[
u_{k+1}(x)=g(x)=u_k(x).
\]
On the other hand, if $x\in \Omega$, we can use the monotonicity of $\mathcal{A}_\veps$ together with the inductive hypothesis, to get
\[
u_{k+1}(x)= \mathcal{A}_\veps[u_{k};f](x)-\veps^2 J_p(f(x)) \geq \mathcal{A}_\veps[u_{k-1};f](x)-\veps^2 J_p(f(x))= u_k(x).
\]

\noindent\textbf{Step 2:} We prove now that for all $k\in \mathbb{N}$ we have that
\begin{align*}
     \left\{\begin{aligned}&  u_{k}(x)  \leq \mathcal{A}_\veps[u_{k};f](x)-\veps^2 J_p(f(x))&\textup{if}& \quad x\in \Omega,\\
    & u_{k}(x) =g(x)&\textup{if}&\quad  x\in\mathbb{R}^d\setminus \Omega.
    \end{aligned}\right.
\end{align*}
By definition we have that $u_k=g$ in $\mathbb{R}^d\setminus \Omega$. Moreover, by Step 1, for every $x\in \Omega$ we have that
\[
u_{k}(x) = \mathcal{A}_\veps[u_{k-1};f](x)-\veps^2 J_p(f(x)) \leq \mathcal{A}_\veps[u_{k};f](x)-\veps^2 J_p(f(x)).
\]

\noindent\textbf{Step 3:} By Lemma \ref{lem:barrier}, there exists $\overline{u}$ such that
\begin{equation*}
\left\{
\begin{aligned}
    \overline{u}(x)&\geq \mathcal{A}_\veps[\overline{u};f](x)-\veps^2J_p(f(x)) &\textup{if}&\quad x\in  \Omega,\\
    \overline{u}(x)&\geq g(x) &\textup{if}& \quad x\in \R^d\setminus\Omega.
\end{aligned}\right.
\end{equation*}
We then have, by the comparison principle stated in Lemma \ref{lem:comparisonDPP}, that $u_k\leq \overline{u}$ in $\mathbb{R}^d$. This implies, by monotonicity of the sequence $u_k$, that there exists $u$ bounded Borel such that
\[
u_k\to u \quad \textup{as} \quad k\to+\infty \quad \textup{pointwise in } \mathbb{R}^d.
\]

\noindent\textbf{Step 4:} We will show now that $u_k\to u$ as 
$k\to+\infty$  uniformly in $\mathbb{R}^d$. We follow ideas from \cite{Lew20}. Note that $u=u_k=g$ in $\mathbb{R}^d\setminus\Omega$ for all $k\in \mathbb{N}$, so we only need to ensure uniform convergence in $\Omega$. First, given $u_m,u_n$ with $m>n$, we have the following estimate for all $x\in \Omega$:
\begin{align*}
|u_{m+1}(x)-u_{n+1}(x)|=& |\mathcal{A}[u_m;f](x)-\mathcal{A}[u_n;f](x)| \\
\leq& \sup_{c\in [m(\veps),M(\veps)]} \Big\{\alpha \sup_{x\in\R^d} |u_m(x)-u_n(x)|+(1-\alpha)\beta \sup_{x\in\R^d} |u_m(x)-u_n(x)| \\
&+  (1-\alpha)(1-\beta) \frac{1}{|B_{\gamma \veps c^{-\frac{\alpha}{2}}}|}\|u_m-u_n\|_{L^1(\Omega)}\Big\}\\
\leq& \eta \sup_{x\in\R^d} |u_m(x)-u_n(x)| + \frac{C_{d,p}}{|B_{\veps M(\veps)^{-\frac{\alpha}{2}}}|}\|u_m-u_n\|_{L^1(\Omega)}\\
\leq& \eta \sup_{x\in\R^d} |u(x)-u_n(x)| + \frac{C_{d,p}}{|B_{\veps M(\veps)^{-\frac{\alpha}{2}}}|}\|u-u_n\|_{L^1(\Omega)},
\end{align*}
where $\eta=\alpha+(1-\alpha)\beta<1$ and we have used in the last step that $u_k$ is a nondecreasing sequence. Passing to the limit as $m\to+\infty$ in the above estimate, we get
\[
|u(x)-u_{n+1}(x)|\leq \eta \sup_{x\in\R^d} |u(x)-u_n(x)| + \frac{C_{d,p}}{|B_{\veps M(\veps)^{-\frac{\alpha}{2}}}|}\|u-u_n\|_{L^1(\Omega)}
\]
which implies
\[
\sup_{x\in\R^d} |u(x)-u_{n+1}(x)|\leq \eta \sup_{x\in\R^d} |u(x)-u_n(x)| + \frac{C_{d,p}}{|B_{\veps M(\veps)^{-\frac{\alpha}{2}}}|}\|u-u_n\|_{L^1(\Omega)}.
\]
Let $$A_{n}:=\sup_{x\in\R^d} |u(x)-u_{n+1}(x)|.$$ Passing to the limit as $n\to+\infty$ in the above estimate, and using the Dominated Convergence Theorem (or the Monotone Convergence Theorem),
\[
\lim_{n\to+\infty} A_{n}=\lim_{n\to+\infty} A_{n+1} \leq \eta\lim_{n\to+\infty} A_{n}, 
\]
which implies that $$\lim_{n\to+\infty} A_{n}=0,$$ and completes the proof.

\noindent\textbf{Step 5:} Finally, we show that $u$ is a solution of \eqref{eq:DPP}. Clearly, $u=g$ in $\mathbb{R}^d\setminus \Omega$. Moreover, it holds that
\begin{align*}
u(x)&= \lim_{k\to+\infty} u_k(x) \\ & = \lim_{k\to +\infty}\mathcal{A}_\veps[u_{k-1};f](x)-\veps^2 J_p(f(x))
\\ &= \mathcal{A}_\veps[\lim_{k\to +\infty} u_{k-1};f](x)-\veps^2 J_p(f(x))\\
&=\mathcal{A}_\veps[u;f](x)-\veps^2 J_p(f(x)).
\end{align*}
where we have used the uniform convergence of Step 4 to interchange the limit and the average $\mathcal{A}_\veps$.
The proof is completed. 
\end{proof}

Let us finally prove that the solutions of \eqref{eq:DPP} are uniformly bounded independently of $\veps>0$. We also find boundary estimates.

\begin{lemma}\label{lem:unifBarr+boundaryesti}
Let $d\in \N$ $p>2$, $\veps\in(0,1)$, $\Omega\subset\R^d$ be an open bounded set, $f\in C(\overline{\Omega})$ and $g\in C_{\textup{b}}(\R^d\setminus\Omega)$. Assume \eqref{as:trunc2}. Let $u_\veps$ be the solution to \eqref{eq:DPP}. Then, there exist $\veps_0>0$ and $T>0$ such that for all $\veps\in(0,\veps_0)$ we have that
\[
\|u_\veps\|_{L^\infty(\R^d)} \leq T,
\]
where $T$ is a positive constant depending on $p$, $\Omega$, $f$ and $g$ (but not on $\veps$). 

Moreover, for all $x_0\in \partial \Omega$, it holds that
\begin{align*}
\limsup_{\veps\to0, y\to x_0} u_\veps (y) \leq g(x_0), \quad \liminf_{\veps\to0, y\to x_0} u_\veps (y)  \geq g(x_0).
\end{align*}
\end{lemma}

\begin{proof}
\textbf{Step 1: } Let us first construct a smooth barrier for the Poisson problem \eqref{eq:ppoisson}.
It is standard to check that the function $$W(x)= |x|^{-\alpha}$$ with $\alpha = \frac{p+d-2}{p-1}$ is radially decreasing and $\Delta_p W(x) = C_{d,p} |x|^{-(\alpha+1)(p-1)-1}$ 
for some positive constant $C_{d,p}$. Let us now take a point $x_0\in \partial\Omega$. By the uniform exterior ball condition, there exist $R>0$ and $z_0\in \R^d\setminus \Omega$ such that $\overline{B_R(z_0)}\cap \overline{\Omega}= \{x_0\}$. Now fix a parameter $\eta>0$ and a constant $K>0$ to be chosen later, and consider the following function
\begin{align*}
    \underline{W}(x)= K (|x-z_0|^{-\alpha}-R^{-\alpha}) + g(x_0)-\eta.
\end{align*}
First, let us note that, for all $x\in \Omega$, we have that
\[
\Delta_p \underline{W}(x) = K^{p-1} \Delta_p W(x-z_0) \geq K^{p-1} C_{d,p}\tilde{R}^{-(\alpha+1)(p-1)-1},
\]
where $\tilde{R}$ is such that $\Omega \subset B_{\tilde{R}}(z_0)$. Thus, we can choose $K$ large enough in such a way that $$\Delta_p \underline{W}(x) \geq \|f\|_{L^\infty(\Omega)}+1$$ for all $x\in \Omega$. On the other hand, if $x\in \overline{\Omega}\subset \R^d\setminus B_R(z_0) $, we have that 
\[
\underline{W}(x) \leq g(x_0)-\eta.
\]
Thus, there exists $r=r(g,\eta)>0$ such that 
\[
\underline{W}(x) \leq g(x)-\frac{\eta}{2} \quad \textup{for all} \quad x\in B_r(x_0)\cap \partial \Omega.
\]
Again, we can choose $K$ large enough such that
\[
\underline{W}(x) \leq g(x)-\frac{\eta}{2} \quad \textup{for all} \quad x\in \partial \Omega \setminus (B_r(x_0)\cap \partial \Omega),
\]
since $|x-x_0|^{-\alpha}-R^{-\alpha}<0$ for all $ x\in \partial \Omega \setminus (B_r(x_0)\cap \partial \Omega) \subset \R^d\setminus B_{R}(z_0)$. Summarizing, we have obtained a function $\underline{W}$ such that
    \begin{align*}
         \left\{
        \begin{aligned}
           \Delta_p \underline{W}(x)&\geq \|f\|_{L^\infty(\Omega)}+1 &\textup{if}& \quad x\in \Omega,\\
            \underline{W}(x)&\leq g(x)-\eta/2 &\textup{if}& \quad x\in \partial\Omega.
        \end{aligned}
        \right.
    \end{align*}
Moreover, $\underline{W}$ is smooth and has a nonvanishing gradient in a neighborhood of $\Omega$. 

\textbf{Step 2:} We will show now that $\underline{W}$ is a uniform-in-$\veps$ barrier for \eqref{eq:DPP}. On one hand, if $f(x)\geq0$, and since $\Delta_p \underline{W}\geq0$ in $\Omega$, we can use the quantitative estimate given in \eqref{pro:asexp-greater2-quantitative} to get
\begin{align*}
    \frac{\mathcal{A}_\veps[\underline{W};f](x) - \underline{W}(x)}{\veps^2} = \frac{\mathcal{A}_\veps^+[\underline{W}](x) - \underline{W}(x)}{\veps^2} = J_p(\Delta_p \underline{W}) +  o_\veps(1) \geq  J_p(\|f\|_{L^\infty(\Omega)}+1) + o_\veps(1) \geq J_p(f(x)),
\end{align*}
where we have used that $o_\veps(1)$ is uniform in $\Omega$. On the other hand, for all $f(x)<0$, we have that $\mathcal{A}_\veps[\underline{W};f](x)=\mathcal{A}_\veps^-[\underline{W}](x)$. In this case, we note that
   \begin{align*}
& \frac{\mathcal{A}_\veps^-[\underline{W}](x)- \underline{W}(x)}{\veps^2}= \!\!\! \sup_{c\in [m(\veps),M(\veps)]}\left\{\alpha \frac{1}{\veps^2}\left(\displaystyle \inf_{B_{\veps^2c^{1-\alpha}}(x)} \underline{W}-\underline{W}(x)\right) + (1-\alpha)\frac{1}{\veps^2}\left( \mathcal{M}_{\veps c^{-\frac{\alpha}{2}}}[\underline{W}](x)-\underline{W}(x) \right)\right\}\\
 &\qquad = \sup_{c\in [m(\veps),M(\veps)]}\left\{\alpha c^{1-\alpha} \left(-\left|\nabla \underline{W}(x)\right| + \frac{o(\veps^2c^{1-\alpha})}{\veps^2c^{1-\alpha}} \right) + (1-\alpha)c^{-\alpha} \left(\nplap \underline{W}(x) + \frac{o(\veps^2 c^{-\alpha})}{\veps^2c^{-\alpha}}\right)\right\}\\
 &\qquad \geq \sup_{c\in [m(\veps),M(\veps)]}\left\{\alpha c^{1-\alpha} \left(-2\left|\nabla \underline{W}(x)\right|\right) + (1-\alpha)c^{-\alpha} \left(\nplap \underline{W}(x)/2 \right)\right\}\\
 &\qquad \geq \alpha m(\veps)^{1-\alpha} \left(-2\left|\nabla \underline{W}(x)\right|\right) + (1-\alpha)m(\veps)^{-\alpha} \left(\nplap \underline{W}(x)/2 \right) \\
 &\qquad \geq 0 > J_p(f(x)).
 \end{align*}
 Thus, by the comparison principle given in  \Cref{lem:comparisonDPP}, we have that $u_\veps \geq \underline{W}$. 
We conclude that $u_\veps$ is uniformly bounded from below.  

 In a similar way, we can check that $u_\veps$ is uniformly bounded from above.

 \textbf{Step 3: } Given $x_0\in \partial \Omega$, we have
 \[
 \liminf_{\veps\to0, \, y\to x_0} u_\veps(y) \geq  \liminf_{\veps\to0, \, y\to x_0}  \underline{W}(y) =  W(x_0) = g(x_0)-\eta.
 \]
Using an analogous argument to the one used to find the lower bound, we can get that 
\[
 \limsup_{\veps\to0, \, y\to x_0} u_\veps(y) \leq    g(x_0)+\eta.
 \]
 The conclusion follows by the arbitrariness of $\eta$.
\end{proof}

\subsection{Convergence as $\veps\to0^+$}

Let us establish, for convenience, the notation 
\begin{equation*}
  \mathcal{S}(\veps, x, t,\varphi) = \left\{
\begin{aligned}
  &\frac{t- \mathcal{A}_\veps[\varphi;f](x)}{\veps^2}+J_p(f(x)) \quad &\textup{if}&\quad  x\in  \Omega,\\
    &t-g(x)\quad  &\textup{if}&\quad  x\in \R^d\setminus\Omega,
\end{aligned}\right.
\end{equation*}
so that problem \eqref{eq:DPP} can be equivalently formulated as
\[
\mathcal{S}(\veps,x, u_\veps(x),u_\veps)=0 \quad \textup{for} \quad x\in \R^d.
\]

First, we observe the following monotonicity and shifting invariance property on $S$, that follows directly from the definition.
\begin{lemma}
    \begin{enumerate}[\rm (a)]
        \item Let $t\in \R$ and $\phi,\psi$ be two bounded Borel functions such that $\phi\leq \psi$ in $\R^d$. Then
        \[
        \mathcal{S}(\veps,x, t,\phi)\geq S(\veps,x, t,\psi).
        \]
        \item Let $t,\xi,\eta\in \R$, and $\phi$ be a bounded Borel function in $\R^d$. Then, for all $x\in \Omega$, we have that
        \[
        \mathcal{S}(\veps,x, t+\xi,\phi+\xi+\eta)=\mathcal{S}(\veps,x, t,\phi)-\frac{\eta}{\veps^2}.
        \]
    \end{enumerate}
\end{lemma}

\begin{proof}[Proof of convergence of solutions to the DPP as $\varepsilon \to 0^+$]
    Let us define 
\begin{align*}
    \overline{u}(x)= \limsup_{\veps\to0, y\to x} u_\veps(y), \quad \textup{and} \quad \underline{u}(x)= \liminf_{\veps\to0, y\to x} u_\veps(y).
\end{align*}
Note that, by definition, it holds that
$$\underline{u} (x)\leq \overline{u}(x).$$ 
Now, if we can show that $\overline{u}$ (resp. $\underline{u}$) is a viscosity subsolution (resp. supersolution)  of the Poisson problem \eqref{eq:ppoisson}, then, by the comparison principle for \eqref{eq:ppoisson}, we get that $\underline{u}\geq \overline{u}$. And thus, we conclude 
$$u_\veps (x) \to u(x)=\underline{u}(x) =\overline{u}(x) $$ as $\veps\to0$ uniformly in $\overline{\Omega}$.

\noindent \textbf{Step 1:} Let us first show that $\overline{u}$ is a viscosity subsolution at points where $f(x_0)\geq0$. We note first that $\overline{u}$ is a bounded upper semicontinuous function. Take $x_0\in \Omega$ and a function $\varphi\in C^\infty_{\textup{b}}(B_R(x_0))$ such that $\varphi(x_0)=\overline{u}(x_0)$ and $\varphi(x)>u(x)$ for all $x\in \Omega$. Our goal is to check that 
\begin{align}\label{eq:subsolproof}
    \Delta_p\varphi(x_0)\geq f(x_0).
\end{align}
It is standard to check that there exists a sequence $(\veps_n,y_n) \to (0,x_0)$ as $n\to+\infty$ in such a way that
\[
u_{\veps_n}(x) - \varphi(x) \leq u_{\veps_n}(y_n) - \varphi(y_n) + e^{-\frac{1}{\veps_n}} \quad \textup{for all} \quad x\in B_{R}(x_0).
\]
Now let $\xi_n\coloneqq u_{\veps_n}(y_{\veps_n})- \varphi(y_n)$. By the properties of the DPP, we get that
\begin{align*}
  0&=\mathcal{S}(\veps_n,y_n, u_\veps(y_n),u_\veps)\\ & =\mathcal{S}(\veps_n,y_n, \varphi(y_n) + \xi_n ,u_\veps) 
  \\ & \geq \mathcal{S}(\veps_n,y_n, \varphi(y_n) + \xi_n ,\varphi + \xi_n + e^{-\frac{1}{\veps_n}})\\
  & \geq  \mathcal{S}(\veps_n,y_n, \varphi(y_n) ,\varphi)- \frac{e^{-\frac{1}{\veps_n}}}{\veps_n^2},
\end{align*}
that is,
\begin{equation}\label{eq:keysupervisc}
    \mathcal{S}(\veps_n,y_n, \varphi(y_n) ,\varphi) \leq \frac{e^{-\frac{1}{\veps_n}}}{\veps_n^2}.
\end{equation}

\textbf{Case 1: } Assume $\Delta_p\varphi(x_0)>0$. We note first that, if $f(x_0)=0$, then $\Delta_p\varphi(x_0)\geq f(x_0)$ trivially. On the other hand, if $f(x_0)>0$, then, there exist $\rho>0$ and $N>0$ such that $f(y_n)>0$,  $\Delta_p\varphi(y_n), |\nabla\varphi(y_n)| >\rho>0$ and $y_n\in \Omega$ for all $n>N$. Then,
\begin{align*}
     \mathcal{S}(\veps_n,y_n, \varphi(y_n) ,\varphi) &= \frac{\varphi(y_n)- \mathcal{A}_{\veps_n}[\varphi;f](y_n)}{\veps_n^2}+J_p(f(y_n)) \\
     &=  \frac{\varphi(y_n)- \mathcal{A}_{\veps_n}^+[\varphi](y_n)}{\veps_n^2}+J_p(f(y_n))\\
     &= -J_p(\Delta_p \varphi(y_n)) + o_{\veps_n}(1) + J_p(f(y_n)),
\end{align*}
where $o_{\veps_n}(1)$ is uniform in $y_n$ by the quantitative estimate given in \Cref{pro:asexp-greater2-quantitative}. Combining the above identity with \eqref{eq:keysubvisc}, we get
\[
-J_p(\Delta_p \varphi(y_n)) + o_{\veps_n}(1) + J_p(f(y_n)) \leq  \frac{e^{-\frac{1}{\veps_n}}}{\veps_n^2}.
\]
Sending $n \to +\infty$ we get that $\Delta_p\varphi(x_0)\geq f(x_0)$, that is what we wanted to prove.

\textbf{Case 2:} Assume now, by contradiction, that $\Delta_p\varphi(x_0)<0$. Then, there exist $\rho>0$ and $N>0$ such that $ |\nabla\varphi(y_n)| >\rho>0$, $\nplap\varphi(y_n)<-\rho$ and $y_n\in \Omega$ for all $n>N$. Assume that there exists a subsequence $y_{n_j}\to x_0$ as $j+\infty$ such that $f(y_{n_j})\geq0$. Then,
\begin{align*}
   -\mathcal{S}(\veps_{n_j},y_{n_j}, \varphi(y_{n_j}) ,\varphi)&= \frac{\mathcal{A}_{\veps_{n_j}}[\varphi;f](y_{n_j})-\varphi(y_{n_j})}{\veps_{n_j}^2}-J_p(f(y_{n_j}))  \\ & =  \frac{\mathcal{A}_{\veps_{n_j}}^+[\varphi](y_{n_j})-\varphi(y_{n_j})}{\veps_{n_j}^2}-J_p(f(y_{n_j}))\\
   &\leq \inf_{c\in[m({\veps_{n_j}}),M({\veps_{n_j}})]}\Big\{
   \alpha c^{1-\alpha}(|\nabla \varphi (y_{n_j})|+\delta)+ (1-\alpha) c^{-\alpha} (\nplap\varphi(y_{n_j}) + \delta)\Big\}\\
   &\leq \inf_{c\in[m(\veps_{n_j}),M(\veps_{n_j})]}\Big\{
   \alpha c^{1-\alpha}(2|\nabla \varphi(y_{n_j})|)+ (1-\alpha) c^{-\alpha} (-\rho/2)\Big\}\\
   & \leq \alpha m(\veps_{n_j})^{1-\alpha}(2|\nabla \varphi(y_{n_j})|)+ (1-\alpha) m(\veps_{n_j})^{-\alpha} (-\rho/2)\\
   &<-1,
\end{align*}
where the last inequality hold for $\veps_{n_j}$ small enough. This is clearly a contradiction with \eqref{eq:keysupervisc}. Note that we have strongly used the quantitative estimates of \Cref{pro:asexp-greater2-quantitative} that are valid away from points where the gradient vanishes. Thus, we must necessary have that, for $\tilde{N}>N$ big enough, the sequence $y_n$ is such that $f(y_n)<0$. In this case, we can use \Cref{pro:asexp-greater2-quantitative} to get
\begin{align*}
     \mathcal{S}(\veps_n,y_n, \varphi(y_n) ,\varphi) &= \frac{\varphi(y_n)- \mathcal{A}_{\veps_n}[\varphi;f](y_n)}{\veps_n^2}+J_p(f(y_n)) \\
     &=  \frac{\varphi(y_n)- \mathcal{A}_{\veps_n}^-[\varphi](y_n)}{\veps_n^2}+J_p(f(y_n))\\
     &= -J_p(\Delta_p \varphi(y_n)) + o_{\veps_n}(1) + J_p(f(y_n)).
\end{align*}
This estimate, together with \eqref{eq:keysupervisc}, yields
\[
\frac{e^{-\frac{1}{\veps_n}}}{\veps_n^2}\geq -J_p(\Delta_p \varphi(y_n)) + o_{\veps_n}(1) + J_p(f(y_n)),
\]
that is, $\Delta_p\varphi(x_0)\geq f(x_0)\geq0$, which is a contradiction with the assumption $\Delta_p\varphi(x_0)<0$.

\textbf{Case 3:} Assume by contradiction that $\plap \varphi(x_0)=0$ and $\nabla \varphi(x_0)=0$. By definition of viscosity solution of the $p$-Laplace Poisson problem (see Appendix \ref{app:viscositydef}), we only need to test with this kind of test function at points where $f(x)>0$. Then, there exist $\rho>0$ and $N>0$ such that $ f(y_n) >\rho>0$ for all $n>N$. Thus, 
\begin{align*}
   -\mathcal{S}(\veps_n,y_n, \varphi(y_n) ,\varphi)&\leq \inf_{c\in[m({\veps_n}),M({\veps_n})]}\Big\{
   \alpha c^{1-\alpha}(|\nabla \varphi (y_n)|+\delta)+ (1-\alpha) c^{-\alpha} K\Big\}- J_p(f(y_n))\\
   &\leq (|\nabla \varphi (y_n)|+\delta)^{\alpha}K^{1-\alpha} - J_p(\rho).
\end{align*}
From here, we can reach again a contradiction with \eqref{eq:keysupervisc} by letting $n\to+\infty$. 

\textbf{Case 4:} Assume that $\plap \varphi(x_0)=0$ and $\nabla \varphi(x_0)\not=0$. On one hand, if $f(x_0)=0$, we trivially have $\Delta_p\varphi(x_0)=0 =f(x_0)$. On the other hand, assume by contradiction that $f(x_0)>0$. We claim that there exists $N>0$ big enough such that $\plap \varphi(y_n)\geq0$. If the claim holds, then we conclude, by \Cref{pro:asexp-greater2-quantitative}, that 
\begin{align*}
     \mathcal{S}(\veps_n,y_n, \varphi(y_n) ,\varphi) &= -J_p(\Delta_p \varphi(y_n)) + o_{\veps_n}(1) + J_p(f(y_n)).
\end{align*}
From here, the conclusion $\Delta_p\varphi(x_0)\geq f(x_0)$ follows using \eqref{eq:keysupervisc}. Let us proof the claim by contradiction. Assume that there exists a subsequence with $y_{n_j}\to x_0$ as $j+\infty$ such that $\Delta_p\varphi(y_{n_j})<0$ (and, thus, $\nplap \varphi(y_{n_j})<0$). Note that there exist $\rho>0$ and $N>0$ such that $f(y_{n_j})>\rho>0$ for all $n_j>N$. Then, we have
\begin{align*}
   -\mathcal{S}(\veps_{n_j},y_{n_j}, \varphi(y_{n_j}) ,\varphi)&= \frac{\mathcal{A}_{\veps_{n_j}}[\varphi;f](y_{n_j})-\varphi(y_{n_j})}{\veps_{n_j}^2}-J_p(f(y_{n_j})) \\ & =  \frac{\mathcal{A}_{\veps_{n_j}}^+[\varphi](y_{n_j})-\varphi(y_{n_j})}{\veps_{n_j}^2}-J_p(f(y_{n_j}))\\
   &\leq \inf_{c\in[m({\veps_{n_j}}),M({\veps_{n_j}})]}\Big\{
   \alpha c^{1-\alpha}(|\nabla \varphi (y_{n_j})|+\delta)+ (1-\alpha) c^{-\alpha} (\nplap\varphi(y_{n_j}) + \delta)\Big\}-J_p(\rho)\\
   &\leq \inf_{c\in[m(\veps_{n_j}),M(\veps_{n_j})]}\Big\{
   \alpha c^{1-\alpha}(2|\nabla \varphi(y_{n_j})|)+ (1-\alpha) c^{-\alpha} \delta\Big\}-J_p(\rho)\\
   & \leq (2|\nabla \varphi(y_{n_j})|)^{\alpha} \delta^{1-\alpha}+ o_{\veps_{n_j}}(1)-J_p(\rho)\\
   &<-J_p(\rho)/2.
\end{align*}
This is clearly a contradiction with \eqref{eq:keysupervisc}.

\textbf{Final comment on Step 1:} When $x_0\in \partial \Omega$ we can use \Cref{lem:unifBarr+boundaryesti} to get that $$\overline{u}(x_0)\leq g(x_0).$$ Thus, we have shown that $\overline{u}$ is a viscosity subsolution of \eqref{eq:ppoisson}.

\noindent\textbf{Step 2:} Let us show first that $\underline{u}$ is viscosity supersolution at points where $f(x_0)\geq0$. We note first that $\overline{u}$ is a bounded lower semicontinuous function. Take $x_0\in \Omega$ and a function $\varphi\in C^\infty_{\textup{b}}(B_R(x_0))$ such that $\varphi(x_0)=\underline{u}(x_0)$ and $\varphi(x)<\underline{u}(x)$ for all $x\in \Omega$. Arguing as in Step 1, we get that there exits a sequence $(\veps_n,y_n) \to (0,x_0)$ as $n\to+\infty$ in such a way that
\begin{equation}\label{eq:keysubvisc}
    \mathcal{S}(\veps_n,y_n, \varphi(y_n) ,\varphi) \geq - \frac{e^{-\frac{1}{\veps_n}}}{\veps_n^2}.
\end{equation}

\textbf{Case 1:} Assume $\Delta_p\varphi(x_0) >0$. Then, there exist $\rho>0$ and $N>0$ such that $\nplap\varphi(y_n), |\nabla\varphi(y_n)| >\rho>0$ and $y_n\in \Omega$ for all $n>N$. Assume that there exists a subsequence $y_{n_j}\to x_0$ as $j+\infty$ such that $f(y_{n_j})<0$. Then,
\begin{align*}
     -\mathcal{S}(\veps_{n_j},y_{n_j}, \varphi(y_{n_j}) ,\varphi) &= \frac{\mathcal{A}_{\veps_{n_j}}[\varphi;f](y_{n_j})- \varphi(y_{n_j})}{\veps_{n_j}^2}-J_p(f(y_{n_j})) \\ & =  \frac{ \mathcal{A}_{\veps_{n_j}}^-[\varphi](y_{n_j})-\varphi(y_{n_j})}{\veps_{n_j}^2}-J_p(f(y_{n_j}))\\
     & \geq \sup_{c\in[m({\veps_{n_j}}),M({\veps_{n_j}})]}\Big\{
   \alpha c^{1-\alpha}(-|\nabla \varphi (y_{n_j})|-\delta)+ (1-\alpha) c^{-\alpha} (\nplap\varphi(y_{n_j}) - \delta)\Big\}\\
   & \geq  \sup_{c\in[m({\veps_{n_j}}),M({\veps_{n_j}})]}\Big\{
   \alpha c^{1-\alpha}(-2|\nabla \varphi (y_{n_j})|)+ (1-\alpha) c^{-\alpha} (\nplap\varphi(y_{n_j})/2)\Big\}\\
   & \geq \alpha m({\veps_{n_j}})^{1-\alpha}(-2|\nabla \varphi (y_{n_j})|)+ (1-\alpha) m({\veps_{n_j}})^{-\alpha} (\nplap\varphi(y_{n_j})/2)\\
   &>1,
\end{align*}
where the last inequality holds for $\veps_{n_j}$ small enough since $\plap \varphi(y_n)>\rho$. This is clearly a contradiction with \eqref{eq:keysubvisc}. Thus, we must necessary have that, for $\tilde{N}>N$ big enough, the sequence $y_n$ is such that $f(y_n)\geq0$.
Then,
\begin{align*}
     \mathcal{S}(\veps_n,y_n, \varphi(y_n) ,\varphi) &= \frac{\varphi(y_n)- \mathcal{A}_{\veps_n}[\varphi;f](y_n)}{\veps_n^2}+J_p(f(y_n)) \\ & =  \frac{\varphi(y_n)- \mathcal{A}_{\veps_n}^+[\varphi](y_n)}{\veps_n^2}+J_p(f(y_n))\\
     &= -J_p(\Delta_p \varphi(y_n)) + o_{\veps_n}(1) + J_p(f(y_n)),
\end{align*}
where $o_{\veps_n}(1)$ is uniform in $y_n$ by the quantitative estimate of \Cref{pro:asexp-greater2-quantitative}. Combining the above identity with \eqref{eq:keysubvisc}, we get
\[
-J_p(\Delta_p \varphi(y_n)) + o_{\veps_n}(1) + J_p(f(y_n)) \geq - \frac{e^{-\frac{1}{\veps_n}}}{\veps_n^2}.
\]
Sending $n \to +\infty$ we get that $\Delta_p\varphi(x_0)\leq f(x_0)$, which is what we wanted to prove. 

\textbf{Case 2:} Assume $\Delta_p\varphi(x_0) \leq0$. In this case, the result is trivial since $\Delta_p\varphi(x_0) \leq0\leq f(x_0)$.

\textbf{Final comment on Step 2:} When $x_0\in \partial \Omega$, we can use \Cref{lem:unifBarr+boundaryesti} to get that $$\underline{u}(x_0)\geq g(x_0).$$ Thus, we conclude that $\underline{u}$ is a viscosity supersolution of \eqref{eq:ppoisson}.

\noindent \textbf{Step 3:} The cases for sub and supersolutions if  $f(x_0)<0$ follow from the above results changing $f$ by $-f$ and $u_\veps$ by $-u_\veps$.

This concludes the proof.
\end{proof}

\section{The associated game}\label{sec:probabilitythings}

As highlighted in \Cref{sec:intro,sect-main}, the Dynamic Programming Principle examined in the previous section relates to a game. In this section we prove, using probabilistic arguments, that the game has a value that coincides with the solution to the DPP. To simplify the presentation, we will only analyze the case $f\geq0$ in $\Omega$ (the necessary adaptations to cover the general case follow similarly). We recall the rules of the game here for convenience:    
\begin{enumerate}[\noindent \rm G(i)]
    \item\label{item1b} Fix a parameter $\veps>0$, an open bounded domain $\Omega$, a starting point $x_0\in \Omega$, a payoff function $g$ defined in $\R^d\setminus\Omega$, and a running payoff function $f$ defined in $\Omega$.
\item\label{item2b} Each turn, given the actual position $x\in \Omega$, the second player (who wants to minimize the final payoff)
 chooses a constant $c \in [m(\veps),M(\veps)]$. Then, the players toss a biased coin with probabilities $\alpha$ for heads and $1-\alpha$ for tails. 
   \begin{enumerate}[$\bullet$]
   \item If the result is heads, the first player (who wants to maximize) chooses the next position at any point in the ball $B_{\veps^2c^{1-\alpha}}(x)$.
   \item If the result is tails, they play a round of tug-of-war with noise in the ball $B_{\gamma \veps c^{-\frac{\alpha}{2}}}(x)$ with probabilities $\beta$ and $1-\beta$ (see description in \Cref{sec:intro}).
   \end{enumerate}
   \item\label{item4b} The process described in \rm{G}\eqref{item2b} is repeated, with a running payoff at each movement of amount $-\veps^2 J_p(f(x))$. The game continues until the position of the game lies outside $\Omega$ for the first time (this position is denoted by $x_\tau$). When this happens, the second player pays the amount ~$g(x_\tau)$ to the first player.
\end{enumerate}

To gain some intuition on why the solution of \eqref{DPP} coincides with the value of the game, let us heuristically make some observations. When tug-of-war with noise is played, the expected outcome is given by
 $$
         \mathcal{M}_{\varepsilon c^{-\frac{\alpha}{2}}}[ u_\veps](x) = \beta\bigg(\frac{1}{2} \sup_{ B_{\gamma \veps c^{-\frac{\alpha}{2}}}(x)}
  u_\veps + \frac{1}{2} \inf_{ B_{\gamma \veps c^{-\frac{\alpha}{2}}}(x)} u_\veps \bigg) + (1-\beta) \fint_{ B_{\gamma \veps c^{-\frac{\alpha}{2}}}(x)}  u_\veps (y) \, \mathrm{d}y,
$$
while when the first player chooses the next position
in $B_{\veps^2c^{1-\alpha}}(x)$ the expected value is given by $$
\sup_{B_{\veps^2c^{1-\alpha}}(x)} u_\veps.$$ 
Then, for a choice of $c\in [m(\veps),M(\veps)]$, we obtain the following expected outcome:
$$
\alpha \sup_{B_{\veps^2c^{1-\alpha}}(x)} u_\veps  + (1-\alpha) \mathcal{M}_{\veps c^{-\frac{\alpha}{2}}}[u_\veps ](x). 
$$
Taking into account that the first player wants to minimize and chooses $c\in [m(\veps),M(\veps)]$,
we finally arrive to the fact that the expected value at $x$ is given by the expected outcome after playing one round of the game, that is, 
$$
u_\veps (x) = \inf_{c\in [m(\veps),M(\veps)]}\left\{\alpha \sup_{B_{\veps^2c^{1-\alpha}}(x)} u_\veps  + (1-\alpha) \mathcal{M}_{\veps c^{-\frac{\alpha}{2}}}[u_\veps](x) 
\right\} - \veps^2 J_p(f(x)).
$$
Thus, we arrive to
the fact that the value function verifies $$u_\veps(x)= \mathcal{A}_\veps[u_\veps;f](x) - \veps^2 J_p(f(x))$$  
for $x\in \Omega$ and we have $u_\veps=g$ outside the domain, that is,
$u_\veps$ solves \eqref{DPP}.

\begin{remark} When $f(x)<0$, a similar reasoning leads to the part of the DPP where
$$
  u_\veps(x)=\sup_{c\in [m(\veps),M(\veps)]}\left\{\alpha \inf_{B_{\veps^2c^{1-\alpha}}(x)} u_\veps + (1-\alpha) \mathcal{M}_{\veps c^{-\frac{\alpha}{2}}}[u_\veps](x) \right\} - \veps^2 J_p(f(x)).
   $$
   \end{remark}
\subsection{Rigorous game-related results}
The game described in {\rm{G}\eqref{item1b}-\rm{G}\eqref{item2b}-\rm{G}\eqref{item4b}} generates a sequence of states
$$
P=\{ x_{0},x_{1},...,x_{\tau}\},
$$
where $x_0$ is the initial position of the game and $x_\tau$ is the final position of the game (the first time that the position of the game lies outside $\Omega$). A \emph{strategy} for Player~I is a function \( S_I \) defined on the partial histories of the game, which specifies the choices that Player~I will make. In particular, it determines the next position of the game when the biased coin toss results in a head, as well as the next position when Player~I is playing tug-of-war. Similarly, a strategy for Player~II is a function \( S_{II} \), also defined on the partial histories, that specifies the choice of \( c\in [m(\veps),M(\veps)] \) and the next position of the game when playing tug-of-war.

When the two players have fixed their strategies \( S_I \) and \( S_{II} \), we can compute the \emph{expected outcome} as follows: The expected payoff, when starting from \( x_0 \) and using the strategies \( S_I \) and \( S_{II} \), is given by
\begin{equation}
\label{eq:defi-expectation}
\mathbb{E}_{S_{I},S_{II}}^{x_0} \left[ -  \varepsilon^2 \sum_{i=0}^{\tau-1} J_p(f(x_i)) + g(x_\tau) \right].
\end{equation}
Here, the expected value is computed according to the probability measure obtained from Kolmogorov's extension theorem. Observe that the game ends almost surely, regardless of the strategies used by the players, that is, \( \mathbb{P} (\tau =+\infty) = 0 \), and therefore the expected value in \eqref{eq:defi-expectation} is well-defined. This is due to the random movements that occur when playing tug-of-war with noise, which forces the position of the game to be out of the domain in a finite number of plays with positive probability.

Then, since Player I tries to maximize the expected outcome and Player II aims to minimize it, the \emph{value of the game for Player I} is given by
\[
u^\varepsilon_I(x_0)=\sup_{S_I}\inf_{S_{II}} \,
\mathbb{E}_{S_{I},S_{II}}^{x_0}\left[- \varepsilon^2 \sum_{i=0}^{\tau-1} J_p(f(x_i))+ g (x_\tau) \right],
\]
while the \emph{value of the game for Player II} is given by 
\[
u^\varepsilon_{II}(x_0)= \inf_{S_{II}}\sup_{S_I} \,
\mathbb{E}_{S_{I},S_{II}}^{x_0}\left[ -  \varepsilon^2\sum_{i=0}^{\tau-1} J_p(f(x_i)) + g (x_\tau) \right]. 
\]
Intuitively, the values $u_I(x_0)$ and $u_{II}(x_0)$ are the best
expected outcomes each player can guarantee when the game starts at
$x_0$.

Notice that it holds that $$u^\varepsilon_I (x) \leq u^\varepsilon_{II}(x).$$ 
If these two values coincide, $u^\varepsilon_I (x) = u^\varepsilon_{II}(x)$, we say that \emph{the game has a value}. Now, we are ready to state and prove our main result in this section.

\begin{theorem} 
\label{teo.valor.DPP} Let the assymptions of \Cref{thm:DPPexsandconv} hold. Then, the game described in {\rm{G}\eqref{item1b}-\rm{G}\eqref{item2b}-\rm{G}\eqref{item4b}} has a value that is characterized by the unique solution  $u_\veps$ to \eqref{eq:DPP}, that is, it holds that 
$$u_\varepsilon(x)=u^\varepsilon_{I}(x)=u^\varepsilon_{II}(x) \quad \textup{for all} \quad  x \in \mathbb{R}^d.$$ 
\end{theorem}

\begin{proof}
    \textbf{Step 1:} Let us prove first that $u_\veps \leq u_I^\veps$. Fix $\delta>0$. Given  any $c \in [m(\veps),M(\veps)]$ we choose a strategy $S_I^*$ for Player I following the solution to the DPP \eqref{DPP} as follows: 
Player I chooses a point $x_{k+1}^I$ such that
\begin{align*}
x_{k+1}^I=S_{I}^{\ast}(x_0,\dots,x_k) \quad \mbox{such that} \quad \sup_{y \in B_{\varepsilon^2 c^{1-\alpha}}(x_k)}u_{\varepsilon}(y)-\frac{\delta}{2^{k+1}}\leq u_{\varepsilon}(x_{k+1}^I),
\end{align*}
which corresponds to a quasi-optimal strategy aiming to maximize $u_\varepsilon$ if the result of the toss is heads. Player I also chooses another point such that 
\begin{align*}
\tilde{x}_{k+1}^I=S_{I}^{\ast}(x_0,\dots,x_k) \quad \mbox{such that} \quad \sup_{y \in B_{\gamma \varepsilon c^{-\frac{\alpha}{2}}}(x_k)}u_{\varepsilon}(y)-\frac{\delta}{2^{k+1}}\leq u_{\varepsilon}(\tilde{x}_{k+1}^I),
\end{align*}
which corresponds to a quasi-optimal strategy when the result of the toss is tails, and thus, they must play tug-of-war with noise.

Given the strategy $S^*_I$ for Player I and any strategy $S_{II}$ (i.e., $(c_{k+1},x^{II}_{k+1})=S_{II}(x_0,\dots,x_k)$) 
 for Player II, we consider the sequence of random variables
$$
M_k=u_{\varepsilon}(x_k) - \varepsilon^2 \sum_{i=0}^{k-1} J_p(f(x_i)) - \frac{\delta}{2^k}.$$ 
Let us see that $(M_k)_{k\geq 0}$ is a \textit{submartingale}. To this end, we need to estimate $$\mathbb{E}_{S_{I}^{\ast},S_{II}}^{x_0}[M_{k+1}|M_0,\dots ,M_k]=\mathbb{E}_{S_{I}^{\ast},S_{II}}^{x_0}[M_{k+1}|x_k].$$ 
It holds that
\begin{align*}
   \mathbb{E}_{S_{I}^{\ast},S_{II}}^{x_0}  [M_{k+1}|x_k]=&   \alpha u_{\varepsilon}({x}_{k+1}^{I})  + (1-\alpha)\left(  \beta \Big(\frac12 u_{\varepsilon}(\tilde{x}_{k+1}^{I})+\frac12 u_{\varepsilon}(x_{k+1}^{II}) \Big) + (1-\beta) \fint_{ B_{\gamma \veps c_{k+1}^{-\frac{\alpha}{2}}} (x_k)} \!\!\!\!\!\! u_{\varepsilon} (y) \, \mathrm{d}y \right)\\
   & -   \varepsilon^2 \sum_{i=0}^{k} J_p(f(x_i)) -\frac{\delta}{2^{k+1}}\\
   \geq & A_\veps^+[u_\veps](x_k)- \left(\alpha+ (1-\alpha)\frac{\beta}{2}\right) \frac{\delta}{2^{k+1}}- 
   \veps^2 \sum_{i=0}^{k} J_p(f(x_i)) -\frac{\delta}{2^{k+1}}\\
   \geq& u_\veps(x_k) + \veps^2 J_p(f(x_k))-\veps^2 \sum_{i=0}^{k} J_p(f(x_i)) -\frac{\delta}{2^{k}}\\
   =& M_k,
\end{align*}
where we have used that $u_\veps$ is a solution of \eqref{DPP} in the last equality. Therefore, $(M_k)_{k\geq 0}$ is \textit{submartingale}.    

Now, using the \textit{Optional Stopping Theorem} (cf. \cite{Williams}), we conclude that, for any $k \in \mathbb{N}$, we have that
\begin{align*}
\mathbb{E}_{S_{I}^{\ast},S_{II}}^{x_0}[M_{\min\{\tau, k\}}]\geq M_0,
\end{align*}
where $\tau$ is the first time such that $x_{\tau}\notin\Omega$. Taking limit $k\rightarrow\infty$, we arrive to
$$
\mathbb{E}_{S_{I}^{\ast},S_{II}}^{x_0}[M_\tau]
= \mathbb{E}_{S_{I}^{\ast},S_{II}}^{x_0} \Big[ 
g(x_\tau) - \varepsilon^2\sum_{i=0}^{\tau-1} J_p(f(x_i)) -\frac{\delta}{2^\tau}\Big] \geq 
u_\varepsilon (x_0) - \delta .
$$
Since the above estimate holds for any strategy $S_{II}$ used by Player II, it holds for the infimum of all possible strategies of Player II. Thus, we get
\begin{align*}
u^\varepsilon_I(x_0)&=\sup_{S_I}\inf_{S_{II}}\,
\mathbb{E}_{S_{I},S_{II}}^{x_0}\left[- \varepsilon^2 \sum_{i=0}^{\tau-1} J_p(f(x_i))
+ g (x_\tau) \right] 
\\
&\geq \inf_{S_{II}}\,
\mathbb{E}_{S_{I}^*,S_{II}}^{x_0}\left[- \varepsilon^2 \sum_{i=0}^{\tau-1} J_p(f(x_i))
+ g (x_\tau) \right]\\
&\geq u_\veps(x_0)-\delta.
\end{align*}
The arbitrariness of $\delta$ concludes the proof.

 \textbf{Step 2:} Let us prove now that $u_\veps \geq u_{II}^\veps$. Fix $\delta>0$. To this end, we define the following quasi-optimal strategy $S_{II}^*$ for Player II: First choose $c^*_{k+1}$ in such a way that
\begin{align*}
\displaystyle 
\mathcal{A}_\veps^+[u_\veps](x_k)  +\frac{\delta}{2^{k+2}} \geq \left\{\alpha \sup_{B_{\veps^2(c^*_{k+1})^{1-\alpha}}(x_k)} u_{\varepsilon}  + (1-\alpha) \mathcal{M}_{\veps (c^*_{k+1})^{-\frac{\alpha}{2}}}[u_{\varepsilon}](x_k)
\right\}.
\end{align*}
and then 
\begin{align*}
x_{k+1}^{II}=S_{II}^{\ast}(x_0,\dots,x_k) \qquad \mbox{such that} \qquad \inf_{y \in B_{\gamma \varepsilon (c^*_{k+1})^{-\frac{\alpha}{2}}}(x_k)}u^{\varepsilon}(y)+\frac{\delta}{2^{k+2}}\geq u^{\varepsilon}(x_{k+1}^{II}).
\end{align*}
Given any strategy $S_I$ for Player I, define the random variables $$
N_k=u_{\varepsilon}(x_k)- \varepsilon^2 \sum_{i=0}^{k-1} J_p(f(x_i)) + \frac{\delta}{2^k}$$ and then let us show that $(N_k)_{k\geq 0}$ is a \textit{supermartingale}. Indeed, we have
\begin{align*}
   \mathbb{E}_{S_{I}^{\ast},S_{II}}^{x_0}  [N_{k+1}|x_k]=&   \alpha u_{\varepsilon}({x}_{k+1}^{I})  + (1-\alpha)\left(  \beta \Big(\frac12 u_{\varepsilon}(\tilde{x}_{k+1}^{I})+\frac12 u_{\varepsilon}(x_{k+1}^{II}) \Big) + (1-\beta) \fint_{ B_{\gamma \veps (c^*_{k+1})^{-\frac{\alpha}{2}}}(x_k)}  u_{\varepsilon} (y) \, \mathrm{d}y \right)\\
   &  -   \varepsilon^2 \sum_{i=0}^{k} J_p(f(x_i)) +\frac{\delta}{2^{k+1}}\\
   \leq & A_\veps^+[u_\veps](x_k)+ \left(1+ (1-\alpha)\frac{\beta}{2}\right) \frac{\delta}{2^{k+2}} -
   \veps^2 \sum_{i=0}^{k} J_p(f(x_i)) +\frac{\delta}{2^{k+1}}\\
   \leq& u_\veps(x_k) + \veps^2 J_p(f(x_k)) - \veps^2 \sum_{i=0}^{k} J_p(f(x_i)) +\frac{\delta}{2^{k}}\\
   =& N_k.
\end{align*}
From here, the rest of the proof of Step 2 follows as in Step 1 to obtain
\begin{align*}
u^\varepsilon_{II}(x_0)&=\inf_{S_{II}} \sup_{S_I}\,
\mathbb{E}_{S_{I},S_{II}}^{x_0}\left[- \varepsilon^2 \sum_{i=0}^{\tau-1} J_p(f(x_i))
+ g (x_\tau) \right] \\
&\leq \sup_{S_I} \,
\mathbb{E}_{S_{I},S_{II}^*}^{x_0}\left[- \varepsilon^2 \sum_{i=0}^{\tau-1} J_p(f(x_i))
+ g (x_\tau) \right]\\
&\leq u_\veps(x_0)+ \delta.
\end{align*} 
Since $\delta$ is arbitrary, this concludes the proof the result.
\end{proof}

\appendix

\section{The geometric mean as infimum of arithmetic means} \label{ApA}
A key tool that allows us to derive asymptotic expansions for the $p$-Laplacian (and related operators) is the fact that any geometric mean can be equivalently expressed as the infimum of certain arithmetic means. More precisely, we rely on the identity given in the following result.
\begin{lemma}\label{lem:ident-crucial}
Let $a,b\geq0$ and $\alpha \in (0,1)$. Then,
\begin{equation}\label{eq:gm-am}
    a^{\alpha}b^{1-\alpha}=\inf_{c>0}\Big\{\alpha c^{1-\alpha} a + (1-\alpha)c^{-\alpha} b\Big\}.
\end{equation}
\end{lemma}
\begin{proof}
    If $a=0$ the result follows since the infimum on the right-hand side of \eqref{eq:gm-am}  is reached as $c\to+\infty$. Now, let us assume that $a>0$. On one hand, by choosing $c=b/a$, we obtain
    \[
    \inf_{c>0}\Big\{\alpha c^{1-\alpha} a + (1-\alpha)c^{-\alpha} b\Big\}\leq \alpha \left(b/a\right)^{1-\alpha} a + (1-\alpha)\left(b/a\right)^{-\alpha} b = a^{\alpha}b^{1-\alpha}.
    \]
The reverse inequality follows from the weighted Arithmetic Mean - Geometric Mean inequality. Specifically,
\[
    \inf_{c>0}\Big\{\alpha c^{1-\alpha} a + (1-\alpha)c^{-\alpha} b\Big\}\geq \inf_{c>0}\Big\{\left(c^{1-\alpha} a\right)^{\alpha} \left(c^{-\alpha} b\right)^{1-\alpha}\Big\} = a^{\alpha}b^{1-\alpha}.
    \]
\end{proof}
Since the results in this paper require performing approximations on quantities within the infimum over 
$c>0$, we need to localize the region over which the infimum is taken and ensure that this truncation does not introduce significant error. More precisely, we have the following result.
\begin{lemma}\label{lem:gm-am-aprox}
Let $a,b\geq0$, $\alpha \in (0,1)$ and $0<m<M<+\infty$. Then,
\begin{equation}\label{eq:gm-am-aprox}
    \left|a^{\alpha}b^{1-\alpha}-\inf_{c\in[m,M]}\Big\{\alpha c^{1-\alpha} a + (1-\alpha)c^{-\alpha} b\Big\}\right| \leq \alpha a \, m^{1-\alpha} + (1-\alpha) b\,  M^{-\alpha}.
\end{equation}
\end{lemma}

\begin{proof}
First, we note that
\begin{align*}
a^{\alpha}b^{1-\alpha} = \inf_{c>0}\Big\{\alpha c^{1-\alpha} a + (1-\alpha)c^{-\alpha} b\Big\} \leq \inf_{c\in[m,M]}\Big\{\alpha c^{1-\alpha} a + (1-\alpha)c^{-\alpha} b\Big\},
\end{align*}
so we only need to prove the reverse inequality. We do this by considering different cases. If $a = b = 0$, then \eqref{eq:gm-am-aprox} holds since $$
\inf_{c\in[m,M]}\Big\{\alpha c^{1-\alpha} a + (1-\alpha)c^{-\alpha} b\Big\} = 0 = a^{\alpha}b^{1-\alpha}.$$
If $a > 0$ and $b = 0$, then
\begin{align*}
\inf_{c\in[m,M]}\Big\{\alpha c^{1-\alpha} a + (1-\alpha)c^{-\alpha} b\Big\} = \inf_{c\in[m,M]}\left\{\alpha c^{1-\alpha} a \right\} = \alpha a \inf_{c\in[m,M]}\left\{ c^{1-\alpha}\right\} = \alpha a m^{1-\alpha},
\end{align*}
and thus, \eqref{eq:gm-am-aprox} holds since $a^{\alpha}b^{1-\alpha} = 0$. The case $a = 0$ and $b > 0$ follows in a similar way. Finally, assume that $a, b > 0$. Let us define $F(c) = \alpha c^{1-\alpha} a + (1-\alpha)c^{-\alpha} b$ and note that $F'(c) = \alpha(1-\alpha) c^{-\alpha} \left(a - b/c\right)$.
This implies that $c = b/a$ is the global minimum of $F$ on $[0,+\infty)$. In particular, if $b/a \in [m, M]$, then
\begin{align*}
\inf_{c\in[m,M]}\Big\{\alpha c^{1-\alpha} a + (1-\alpha)c^{-\alpha} b\Big\} = \inf_{c>0}\Big\{\alpha c^{1-\alpha} a + (1-\alpha)c^{-\alpha} b\Big\} = a^{\alpha}b^{1-\alpha}.
\end{align*}
On the other hand, if $b/a < m$, then
\begin{align*}
\inf_{c\in[m,M]}\Big\{\alpha c^{1-\alpha} a + (1-\alpha)c^{-\alpha} b\Big\} &= \alpha m^{1-\alpha} a + (1-\alpha)m^{-\alpha}b 
\\ & \leq \alpha m^{1-\alpha} a + (1-\alpha)\left(\frac{b}{a}\right)^{-\alpha} b \\
&= \alpha m^{1-\alpha} a + a^{\alpha}b^{1-\alpha}.
\end{align*}
Finally, if $b/a > M$, then
\begin{align*}
\inf_{c\in[m,M]}\Big\{\alpha c^{1-\alpha} a + (1-\alpha)c^{-\alpha} b\Big\} &= \alpha M^{1-\alpha} a + (1-\alpha)M^{-\alpha}b 
\\ & \leq \alpha \left(\frac{b}{a}\right)^{1-\alpha} a + (1-\alpha)M^{-\alpha} b \\
&= a^{\alpha}b^{1-\alpha} + (1-\alpha) M^{-\alpha} b,
\end{align*}
which concludes the proof.
\end{proof}

\section{Viscosity solutions for the $p$-Laplacian  and the mean value properties}\label{app:viscositydef}
We devote this appendix to discuss the notion of viscosity solutions for the $p$-Laplacian and also for the asymptotic mean value properties used in this paper.

\subsection{Viscosity solutions of $p$-Laplace equations for $p\in(2,\infty)$}
We adopt the following notion of viscosity solution for the $p$-Laplacian.

\begin{definition}[Viscosity solution of $p$-Laplace equation]\label{def:viscsol-p-Laplace}
Let $p\in(2,\infty)$ and $\Omega\subset\R^d$ be an open set. Suppose that $f$ is a continuous function in $\Omega$. We say that a bounded lower (resp. upper) semicontinuous function $u$ in $\Omega$ is a \emph{viscosity supersolution} (resp. \emph{subsolution}) of
\[
\Delta_p u =f \quad \textup{in} \quad \Omega
\]
if the following holds: 
Whenever $x_0\in \Omega$ and $\varphi\in C^2(B_R(x_0))$ for some $R>0$  are such that 
\begin{align}\label{eq:visctouch-super}
\varphi(x)-u(x)&\leq \varphi(x_0)-u(x_0) \quad \textup{for all} \quad x\in (B_R(x_0)\setminus\{x_0\})\cap\Omega,\\\label{eq:visctouch-sub}
\text{(resp. }\varphi(x)-u(x)&\geq \varphi(x_0)-u(x_0))
\end{align}
with $\nabla\varphi(x_0)\not=0$ if $f(x_0)=0$, then we have
\begin{align}\label{eq:viscsuper}
    \Delta_p\varphi(x_0) &\leq f(x_0).\\\label{eq:viscsub}
   \text{(resp. } \Delta_p\varphi(x_0) &\geq f(x_0).)
\end{align}

A \emph{viscosity solution} is a bounded continuous function $u$ in $\Omega$ that is both a viscosity supersolution and a viscosity subsolution.
\end{definition}

\begin{remark}\label{rem:novanishgrad}
    Note that we do not need to test with test functions with vanishing gradient when the right-hand side is zero. This is due to the fact that, if $\nabla\varphi(x_0)=0$, then $\Delta_p\varphi(x_0)=0$ for $p>2$, and thus \Cref{eq:viscsub,eq:viscsuper} are trivially satisfied.
\end{remark}

Actually, it is very well known that the class of test functions $\varphi$ in \Cref{def:viscsol-p-Laplace} can be restricted to a smaller class without losing any generality. We recall several simplifications that are useful in our setting. We present all of them in the context of supersolutions, since the results for subsolutions follow in the same way.

\begin{remark}[Strict contact points]
    It is well known that condition \eqref{eq:visctouch-super} can be replaced by
    \begin{align*}
        \varphi(x_0)=u(x_0) \quad \textup{and} \quad  \varphi(x)&< u(x) \quad \textup{for all} \quad x\in (B_R(x_0)\setminus\{x_0\})\cap\Omega.
    \end{align*}
\end{remark}

\begin{remark} [More regular test functions]\label{rem:regulartest}
 Without loss of generality, we can assume that  $\varphi\in C^\infty(B_R(x_0))$. Indeed, let us consider a test function $\varphi\in C^2(B_R(x_0))$ for viscosity supersolutions (i.e. satisfying \eqref{eq:visctouch-super}), a smooth mollifier $\rho_\delta\in C_c^\infty(\overline{B_\delta(0)})$ and the mollified version of $\varphi$ given by $\varphi_\delta=\rho_\delta*\varphi$. Let us recall that $
    \varphi_\delta \to \varphi$, $\nabla\varphi_\delta \to \nabla\varphi$ and $D^2\varphi_\delta \to D^2\varphi$ as $\delta\to0^+$ uniformly in $B_{R/2}(x_0)$. By uniform convergence of $\varphi_\delta$, there exists a sequence $\{x_\delta\}_{\delta>0}$ such that $x_\delta\to x_0$ as $\delta\to0$ such that
\begin{align*}
\varphi_\delta(x)-u(x)&\leq \varphi_\delta(x_\delta)-u(x_\delta) \quad \textup{for all} \quad x\in (B_{R/2}(x_\delta)\setminus\{x_\delta\})\cap\Omega.
\end{align*}
By definition of viscosity supersolution, we have that $\Delta_p\varphi_\delta(x_\delta) \leq f(x_\delta)$, 
which implies \eqref{eq:viscsuper}, by uniform convergence of $D^2 \varphi_\delta$ and $\nabla \varphi_\delta$ to $D^2 \varphi$ and $\nabla \varphi$ respectively, and the continuity of $f$.
\end{remark}

\begin{remark}[Non vanishing $p$-Laplacian]\label{rem:novanplap} 
    Without loss of generality, we can assume that $\varphi$ is such that, whenever $\nabla \varphi(x_0)\not=0$, then  $ \Delta_p \varphi(x_0)\not=0$. First note that $|\nabla \varphi(x_0)|^{p-2}$ is well defined for all $p\in(1,\infty)$ since $\nabla \varphi(x_0)\not=0$, and, thus, 
    \[
    \Delta_p\varphi(x_0)= |\nabla \varphi(x_0)|^{p-2}\nplap \varphi(x_0)=|\nabla \varphi(x_0)|^{p-2} \left(\Delta \varphi(x) + (p-2) \left<D^2\varphi(x_0) \frac{\nabla\varphi(x_0)}{|\nabla\varphi(x_0)|}, \frac{\nabla\varphi(x_0)}{|\nabla\varphi(x_0)|}\right>\right).
    \]
    Let us assume that $\varphi$ is a test function for supersolutions at $x_0$, that is, 
\begin{align*}
     \varphi(x_0)=u(x_0) \quad \textup{and} \quad  \varphi(x)&< u(x) \quad \textup{for all} \quad x\in (B_R(x_0)\setminus\{x_0\})\cap\Omega,
\end{align*}
and such that $\Delta_p\varphi(x_0)=0$. Since $\nabla \varphi(x_0)\not=0$, then we must have that $\nplap \varphi(x_0)=0$. Now consider the function 
\[
\varphi_\delta(x)=\varphi(x)-\delta|x-x_0|^2.
\]
Clearly,
\begin{align*}
     \varphi_\delta(x_0)=u(x_0) \quad \textup{and} \quad  \varphi_\delta(x)&< u(x) \quad \textup{for all} \quad x\in (B_R(x_0)\setminus\{x_0\})\cap\Omega,
\end{align*}
and then $\Delta_p\varphi_\delta (x_0)\leq f(x_0)$. Moreover, note that $\nabla \varphi_\delta(x_0)=\nabla\varphi(x_0)$ and $D^2 \varphi_\delta(x_0)=D^2\varphi(x_0)- 2\delta I$. Then, 
\begin{align*}
\Delta_p^{\textup{N}}\varphi_\delta(x_0)&=\Delta\varphi_\delta(x_0) + (p-2)\left<D^2\varphi_\delta(x_0) \frac{\nabla\varphi_\delta(x_0)}{|\nabla\varphi_\delta(x_0)|}, \frac{\nabla\varphi_\delta(x_0)}{|\nabla\varphi_\delta(x_0)|}\right>\\
&=(\Delta\varphi(x_0)-2\delta d) + (p-2)\left<\left(D^2\varphi(x_0)- 2\delta I\right) \frac{\nabla\varphi(x_0)}{|\nabla\varphi(x_0)|}, \frac{\nabla\varphi(x_0)}{|\nabla\varphi(x_0)|}\right>\\
&= \Delta_p^{\textup{N}}\varphi(x_0) -  2(d+p-2) \delta\\
&= -  2(d+p-2) \delta<0.
\end{align*}
Thus, we get that $\Delta_p\varphi_\delta(x_0)= -2(d+p-2)|\nabla\varphi(x_0)|^{p-2} \delta <0$.
Finally, we observe that
\begin{align*}
f(x_0)\geq \Delta_p\varphi_\delta(x_0) &= |\nabla\varphi_\delta(x_0)|^{p-2} \Delta_p^{\textup{N}}\varphi_\delta(x_0) =|\nabla\varphi(x_0)|^{p-2} \left(\Delta_p^{\textup{N}}\varphi(x_0) - 2(d+p-2) \delta\right)\\
&=\Delta_p\varphi(x_0)   -2(d+p-2) |\nabla\varphi(x_0)|^{p-2} \delta.
\end{align*}
Since we can take $\delta>0$ arbitrarily small, we conclude that $\Delta_p\varphi(x_0)\leq f(x_0)$, as we wanted to check. 
\end{remark}

Let us define now the concept of viscosity solutions for the boundary value problem
    \begin{align}\label{eq:ppoisson-appendix}
        \left\{
        \begin{aligned}
           \Delta_p u(x)&= f(x) &\textup{if}& \quad x\in \Omega,\\
            u(x)&= g(x) &\textup{if}& \quad x\in \partial\Omega.
        \end{aligned}
        \right.
    \end{align}

\begin{definition}[Viscosity solution of $p$-Laplace boundary value problem]\label{def:viscsol-p-Laplace-poisson}
Let $p\in(2,\infty)$ and $\Omega\subset\R^d$ be a bounded open set. Suppose that $f$ is a continuous function in $\overline{\Omega}$ and that $g$ is a continuous function in $\partial \Omega$. We say that a bounded lower (resp. upper) semicontinuous function $u$ in $\overline{\Omega}$ is a \emph{viscosity supersolution} (resp. \emph{ subsolution}) of
\eqref{eq:ppoisson-appendix}
if the following holds: 
\begin{enumerate}[\rm (a)]
    \item $u$ is a viscosity supersolution (resp. subsolution) of $\Delta_p u=f$ in $\Omega$ (in the sense of \Cref{def:viscsol-p-Laplace})
    
    \medskip
    
    \item $u(x)\geq g(x)$ (resp. $u(x)\leq g(x)$) for all $x\in \partial \Omega$.
\end{enumerate}
 \end{definition}

\subsection{Asymptotic expansions in the viscosity sense for $p\in(2,\infty)$}

Having in mind all the considerations of the previous section about how restrictive test functions can be taken, we propose a definition of viscosity solution of the asymptotic expansion problems under minimal requirements on the test functions. 
Let us recall the notation
\[
 \overline{\mathcal{A}}_\veps[\varphi;f](x)\coloneqq\left\{\begin{aligned}
 \mathcal{A}_\veps^+[\varphi](x) \quad \textup{if} \quad f(x)\geq0,\\
  \mathcal{A}_\veps^-[\varphi](x) \quad \textup{if} \quad f(x)<0,
 \end{aligned}\right. \quad \textup{and} \quad   \underline{\mathcal{A}}_\veps[\varphi;f](x)\coloneqq\left\{\begin{aligned}
 \mathcal{A}_\veps^+[\varphi](x) \quad \textup{if} \quad f(x)>0,\\
  \mathcal{A}_\veps^-[\varphi](x) \quad \textup{if} \quad f(x)\leq 0.
 \end{aligned}\right.
\]

\begin{definition}[The asymptotic mean value property in the viscosity sense]\label{def:asMVPviscosity}
Let $p\in(2,\infty)$, $\Omega\subset\R^d$ be an open set and $\mathcal{A}_{\veps}$ be given by either $\overline{\mathcal{A}}_\veps$ or $\underline{\mathcal{A}}_\veps$.
  Suppose that $f$ is a continuous function in $\Omega$.
We say that a bounded lower (resp. upper) semicontinuous function $u$ in $\Omega$ is a \emph{viscosity supersolution} (resp. \emph{ subsolution}) of
\begin{align*}
     \mathcal{A}_\veps[u;f]-u=\veps^2 J_p(f) + o(\veps^2) \quad \textup{in} \quad \Omega
\end{align*}
if the following holds: Whenever $x_0\in \Omega$ and $\varphi\in C^\infty(B_R(x_0))$ for some $R>0$ are such that \eqref{eq:visctouch-super} (resp. \eqref{eq:visctouch-sub}) holds with $\nabla\varphi(x_0)\not=0$ and $\Delta_p\varphi(x_0)\not=0$ if $f(x_0)=0$, 
then we have
\begin{align*}
 \mathcal{A}_\veps[\varphi;f](x_0)-\varphi(x_0)&\leq \veps^2 J_p(f(x_0)) + o(\veps^2).
    \\
   \text{(resp. }  \mathcal{A}_\veps[\varphi;f](x_0)-\varphi(x_0)&\geq \veps^2 J_p(f(x_0)) + o(\veps^2).)
\end{align*}

A \emph{viscosity solution} is a continuous function $u$ in $\Omega$ that is both a viscosity supersolution and a viscosity subsolution.
\end{definition}

\begin{remark}
   Note that we have the extra assumptions  $\nabla\varphi(x_0)\not=0$ and $\Delta_p\varphi(x_0)\not=0$ if $f(x_0)=0$ on the definition of viscosity solution for the asymptotic mean value property. This is motivated from the definition of viscosity solution of the $p$-Laplacian problem together with \Cref{rem:novanishgrad,rem:novanplap}. We have also asked $\varphi\in C^\infty(B_R(x_0))$ instead of just $\varphi\in C^2(B_R(x_0))$ following \Cref{rem:regulartest}.
\end{remark}

\section*{Acknowledgments}
	
	We want to thank J. J. Manfredi for several nice discussions that helped us to improve our results.
	
	This work was started during a visit of J. D. Rossi to Universidad Autónoma de Madrid and ICMAT for the thematic semester ``Diffusion, Geometry, Probability and Free Boundaries"  and 
	finished during a visit of F. del Teso to Universidad de Buenos Aires funded by the "José Castillejo" scholarship with reference CAS23/00051. The authors are grateful to both
	institutions for the friendly and stimulating working atmosphere. 
	
	F. del Teso was partially supported by  the Spanish Government through  RYC2020-029589-I, PID2021-
127105NB-I00 and CEX2019-000904-S funded by the MICIN/AEI.
	
	 J. D. Rossi was partially supported by 
			CONICET PIP GI No 11220150100036CO
(Argentina), PICT-03183 (Argentina) and UBACyT 20020160100155BA (Argentina).

\end{document}